\RequirePackage{fix-cm}
\documentclass{article}
\usepackage{longtable}

\usepackage{arxiv}
\usepackage{lipsum}
\usepackage{placeins}
\usepackage[italicComments=true,indLines=true,noEnd=false]{algpseudocodex}
\usepackage{float}
\usepackage[english]{babel}
\usepackage[utf8]{inputenc} 
\usepackage[T1]{fontenc}    
\usepackage{hyperref}       
\usepackage{url}            
\usepackage{booktabs}       
\usepackage{stackrel,amssymb}
\usepackage{bm}
\usepackage{amsfonts,amsmath,amstext,amsbsy,amsthm}
\usepackage{nicefrac}       
\usepackage[version=4]{mhchem}
\usepackage{microtype}
\usepackage{mathrsfs}
\usepackage{graphicx}
\usepackage[numbers,sort&compress]{natbib}
\usepackage{doi}
\newtheorem{prop}{Proposition}

\usepackage{xcolor}
\usepackage{enumitem}
\usepackage{caption}
\usepackage[title]{appendix}
\usepackage[capitalise,nameinlink]{cleveref}

\usepackage[bottom]{footmisc}
\usepackage{multirow}
\usepackage{subcaption}
\usepackage{pgfplots}
\usepackage{tikz}
\usepackage{comment}
\usepackage[table]{xcolor} 
\usepackage{tabularx}                   
\newcolumntype{C}{>{\centering\arraybackslash}X}
\newcolumntype{L}{>{\raggedright\arraybackslash}X}
\usepackage{algorithm}
\usepackage{algpseudocodex}
\usetikzlibrary{positioning, arrows.meta}

\definecolor{ao(english)}{rgb}{0.0, 0.5, 0.0}

\pgfplotsset{compat=1.17}

\title{PDE-Free Mass-Constrained Learning of Complex Systems with Hidden States}

\author{
\textbf{Gianmaria Viola\textcolor{blue}{$^{1}$}, Alessandro Della Pia\textcolor{blue}{$^{2}$}, Ioannis Kevrekidis\textcolor{blue}{$^{3}$}, Lucia Russo \textcolor{blue}{$^{1,*}$}, Constantinos Siettos\textcolor{blue}{$^{4,}$}\thanks{Corresponding authors, emails: \texttt{lucia.russo@stems.cnr.it, constantinos.siettos@unina.it}}}
{}\\
\textcolor{blue}{$^{(1)}$}Institute of Science and Technology for Energy and Sustainable Mobility (STEMS), \\ \emph{Consiglio Nazionale delle Ricerche (CNR)}, Naples 80125, Italy\\
\textcolor{blue}{$^{(2)}$}Modeling and Engineering Risk and Complexity, \\ \emph{Scuola Superiore Meridionale, School for} \emph{Advanced Studies}, Naples 80138, Italy \\
\textcolor{blue}{$^{(3)}$}Department of Chemical and Biomolecular Engineering, Department of Applied Mathematics and \\
\hspace{0.39cm}Statistics \& Department of Urology, \emph{Johns Hopkins University}, Baltimore, MD 21218, USA\\
\textcolor{blue}{$^{(4)}$}Dipartimento di Matematica e Applicazioni ‘‘Renato Caccioppoli", \\ \emph{Universit\`a degli Studi di Napoli} \emph{‘‘Federico II"}, Naples 80126, Italy\\
}
\pgfplotsset{compat=1.17}

\date{\today}

\begin{document}

\maketitle

\begin{abstract}

We propose a three-tier machine learning (ML) framework based on the so-called \textit{next-generation Equation-free} algorithms for learning the spatio-temporal dynamics of mass-constrained complex systems with hidden states, whose dynamics can in principle be described by partial differential equations (PDEs), but lack explicit models. 
In the first step, we employ Diffusion Maps (DMs), a nonlinear manifold learning algorithm, to extract low-dimensional latent representations of the complex spatio-temporal evolution. In the second step, we learn manifold-informed reduced-order models (ROMs) with Sparse Identification of Nonlinear Dynamics (SINDy) and standard linear Multivariate Autoregressive models (MVARs) to approximate the solution operator on the latent space. In the final step, the latent dynamics are lifted back to the original high-dimensional space by solving an (ill-posed) pre-image problem via a convex interpolation based on the $k$-NN algorithm. In doing so, the proposed framework reconstructs the solution operator of the unknown mass-constrained PDE, without explicitly identifying the PDE itself. For comparison purposes, we also evaluated the performance of the scheme for constructing ROMs based on Proper Orthogonal Decomposition (POD) and prove that both POD and the $k$-NN lifting operators preserve the mass. We illustrate the approach using two benchmark problems: (a) the Hughes model of crowd dynamics, which minimizes walking time while avoiding obstacles and high-density regions, and (b) a CFD problem involving the spatio-temporal evolution of a passive tracer advected by a periodic Navier--Stokes velocity field. We show that ROMs informed by DMs yield parsimonious models that consistently outperform the best POD-informed ROMs, yielding stable and accurate approximations of the solution operator in the latent space and, via reconstruction, in the original high-dimensional space over long time horizons.

\end{abstract}

\section{Introduction}
\label{sec:intro}

Many real-world systems exhibit complex spatio-temporal dynamics that, in principle, can be described by partial differential equations (PDEs) but also depend on hidden variables~\cite{kemeth2018emergent,reinbold2019data,kemeth2022learning,sroczynski2024learning} that are not directly observable. These hidden variables often represent unknown closures, as in the case of computational fluid dynamics (CFD)~\cite{russo2007reduced,kutz2017deep,raissi2020hidden,joglekar2023machine,evangelou2022parameter,Kevrekidis_DM}, or internal states, as in biological systems~\cite{engelhardt2017bayesian,yazdani2020systems,lee2023learning,ahmadi2024ai}, or optimal transport strategies—such as those governing crowd dynamics~\cite{bellomo2011modeling,bellomo2012modeling,cristiani2014multiscale,cristiani2014overview,bellomo2022towards,bellomo2023human}. For modeling such systems, one may consider a general data-driven setting in which the specific functional forms or entire operators of the macroscopic PDEs governing the emergent crowd dynamics are not known explicitly. The goal in this case is to use data, either from high-fidelity microscopic simulations or from real-world measurements, to approximate the right-hand sides or closures of the effective macroscopic equations within a finite-dimensional space. This can be achieved through black-box or gray-box surrogate models, typically implemented as (deep) neural networks (DNNs)~\cite{Lee2020,Galaris2022,lee2023learning,fabiani2024task,de2024ai}. Alternative machine learning schemes, such as convolutional neural networks (CNNs)~\cite{qu2022learning,vlachas2022multiscale} and Gaussian processes~\cite{pang2020physics,chen2021solving}, have also been explored in similar contexts. 

However, approximating the right-hand side of PDEs in finite spaces using DNNs poses various challenges and limitations. For example, the accurate approximation of spatial and temporal derivatives requires access to high-resolution samples, and additional pre-processing is needed in the case of noisy data~\cite{Lee2020,lee2023learning}. Moreover, in real-world scenarios, data are typically available as discrete-time snapshots, limiting the accuracy of temporal derivative approximations. This can introduce significant modeling biases and numerical instabilities.  
Another approach is to approximate the integral evolution (solution) operator of PDEs, as defined by the corresponding Cauchy problem, using Neural Operators (NOs)~\cite{li2023deep,goswami2023physics,li2020fourier,peyvan2024riemannonets,zappala2024learning,fabiani2025randonets}. NOs approximate continuous evolution operators of PDEs in infinite-dimensional spaces, theoretically enabling mesh-independent solutions and broad generalization. However, the NO paradigm requires learning a highly nonlinear mapping between function spaces, which can demand large training datasets and may suffer from the curse of dimensionality. 

To circumvent the ``curse of dimensionality'' inherent in the training of surrogate DNN- or NO-based PDE models, learning Reduced Order Models (ROMs)~\cite{schilders2008model,quarteroni2014reduced} in well-constructed latent spaces is crucial to significantly lowering computational cost while capturing the essential dynamics using surrogate DNNs, Gaussian Processes, or NOs~\cite{papaioannou2022time,kemeth2022learning,evangelou2022double,kontolati2024learning,fabiani2024task}. These ROMs must also inherently preserve fundamental principles, such as stability for long-run simulations and ``physical'' accuracy when their dynamics are lifted back to the reconstructed space.

Within this context, we propose a fully data-driven \textit{``PDE-free''} machine learning framework based on nonlinear manifold learning and in particular Diffusion Maps \cite{coifman2005geometric,coifman2006diffusion}, inspired by our previous work on the \textit{Equation-Free multiscale framework}~\cite{kevrekidis2003equation} and \textit{``Next-Generation Equation-Free''} algorithms~\cite{arbabi2021particles,kemeth2022learning,papaioannou2022time,Patsatzis_2023,evangelou2022parameter,Kevrekidis_DM,koronaki2024nonlinear,alvarez2025next,sroczynski2024learning}. The proposed framework aims to learn the solution operator of the unavailable PDE for high-dimensional systems involving hidden or unobservable dynamics, without explicitly constructing a surrogate PDE per se, but by reconstructing the spatio-temporal dynamics based on ROMs in latent spaces learned with SINDy and with vanilla linear MVARs. We illustrate the methodology via two benchmark problems: (i) the Hughes PDE for crowd dynamics, (ii) the spatio-temporal evolution of a passive scalar under advection by a Navier–Stokes flow. The objective is to learn appropriate \textit{parsimonious} low-dimensional \textit{manifold-informed} ROMs that can implicitly represent hidden dynamics in latent spaces, possibly by augmenting observables with delayed coordinates thus exploiting Takens’ and Whitney’s embedding theorems~\cite{sauer1991embedology}. Spatio-temporal dynamics, usually modeled by PDEs, evolve on infinite-dimensional state spaces, but their dynamics often collapse to finite-dimensional attractors. Therefore, by applying time-delay embeddings to latent spaces, Whitney’s/Takens’ theorem establishes a diffeomorphic equivalence between the delay-coordinate latent model and the original system dynamics (see results and discussion in~\cite{sauer1991embedology,rico1992discrete,krischer1993model,shvartsman1998low,shvartsman2000order,papaioannou2022time,dylewsky2022principal,kemeth2022learning,axaas2023model,gallos2024data,patil2025separation}). In fact, while Takens’ theorem guarantees that a single generic scalar observable with sufficiently many delays can reconstruct the emergent dynamics, Whitney-type embedding theorems (see~\cite{sauer1991embedology}) indicate that a sufficiently expressive multivariate representation can result in fewer delays (even only a single delay per variable).
This idea for constructing such surrogate models can be traced back to the 1990s~\cite{deane1991low,rico1992discrete,krischer1993model,temam1995navier,graham1996alternative,shvartsman1998low,shvartsman2000order}, where POD was used to construct ROMs from spatio-temporal simulations of PDEs, while for large-scale complex systems, it was encoded in the \textit{Equation-Free multiscale framework}~\cite{kevrekidis2003equation}, bypassing the need to construct models in explicit form to enable numerical analysis tasks such as bifurcation and stability analysis for macroscopic (emergent) dynamics. In particular, to our knowledge, the first time that nonlinear ODEs were learned using NNs in latent spaces (discovered using autoencoders) based on time delays was in~\cite{rico1992discrete}—and it was performed on experimental electrodissolution data; while, the first time that this was done for distributed systems (PDEs) in latent spaces, using POD and time delays of POD coordinates, was in~\cite{krischer1993model}, in which the data were experimental spatiotemporal movies of PEEM of CO oxidation on Pt. The theoretical link of the spatiotemporal work with the theory of Approximate Inertial Manifolds~\cite{titi1990approximate,temam1995navier,constantin2012integral,koronaki2024nonlinear}, which underpins the existence of low-dimensional models for infinite-dimensional spatiotemporal systems, was used in~\cite{shvartsman1998low} and especially in~\cite{shvartsman2000order}, where NNs were used to identify the nonlinear Galerkin equations on the PDE inertial manifold. In recent years, this approach has been extended through the application of other manifold learning techniques such as Laplacian Eigenmaps~\cite{belkin2003laplacian,belkin2008towards}, Diffusion Maps (DMs)~\cite{coifman2005geometric,coifman2006diffusion,coifman2008diffusion,kemeth2018emergent,dsilva2018parsimonious,soize2022probabilistic,soize2024transient,Patsatzis_2023,chin2024enabling}, ISOMAP~\cite{tenenbaum2000global,bollt2007attractor}, and autoencoders (also introduced in the 1990s~\cite{kramer1991nonlinear}), combined with techniques for solving the pre-image problem~\cite{coifman2006geometric,chiavazzo2014reduced,papaioannou2022time,evangelou2022double,Patsatzis_2023}, and with surrogate machine learning models~\cite{papaioannou2022time,charalampopoulos2022uncertainty,vlachas2022multiscale,floryan2022data,conti2023reduced,gallos2024data,koronaki2024nonlinear,constante2024data,fabiani2024task,faraji2025shallow}, including multivariate autoregressive models (MVARs) and Gaussian Processes ~\cite{papaioannou2022time}, SINDy~\cite{brunton2016discovering,Rudy2017,conti2023reduced}, Koopman Operator \cite{gallos2024data}, Neural Operators \cite{kontolati2024learning}, and Operator inference for handling non-polynomial nonlinearities~\cite{Benner2020, Kramer2024}.

Our method proceeds in three main steps. In the first step, we identify the intrinsic latent coordinates using nonlinear manifold learning (DMs). For comparison, 
POD, widely used in the CFD community, is also employed to highlight the advantages of DMs. In the second step, we learn the latent dynamics via surrogate models (here for our illustrations via SINDy but also with vanilla linear MVARs with delayed coordinates.) In the third step, we reconstruct the spatio-temporal dynamics by lifting the learned latent dynamics back into the high-dimensional physical space. While POD lifting reduces to a straightforward linear projection, lifting from DMs (or from any other nonlinear manifold learning algorithm) involves solving an ill-posed nonlinear pre-image problem, which we address here through a $k$-nearest neighbors ($k$-NN)-based decoder with convex interpolation~\cite{chin2024enabling}. First, we prove that both the linear POD and $k$-NN lifting operators preserve the mass. We then demonstrate that DMs-informed ROMs yield parsimonious models that consistently outperform the best POD-informed ROMs. Remarkably, SINDy achieves this using only a single time delay, thereby demonstrating that DMs provide a rich representation of the embedded manifold. Overall, the proposed framework provides stable and accurate approximations of the solution operator in the latent space and, through reconstruction, in the original high-dimensional space over long time horizons.

The paper is organized as follows. In Section~\ref{sec:three-tier}, we present the proposed methodology. We first prove that both the POD lifting operator preserves the mass in the reconstructed space. We then briefly describe the DMs nonlinear manifold learning algorithm and prove that the corresponding $k$-NN (nonlinear) decoder preserves both mass and positivity of the reconstruction state. Manifold-informed MVARs and SINDy surrogate models are presented in Sections~\ref{subsec:multivariate} and \ref{subsec:SINDy}. In Section~\ref{sec:casestudies}, we describe the two benchmark case studies. 
Numerical results are presented in Section~\ref{sec:results}, and conclusions are drawn in Section~\ref{sec:conclusions}. 
More details on benchmark models, and numerical simulations are provided in Appendices~A-D.

\section{The Methodology}
\label{sec:three-tier}

\begin{figure}
    \centering  \includegraphics[width=1\linewidth]{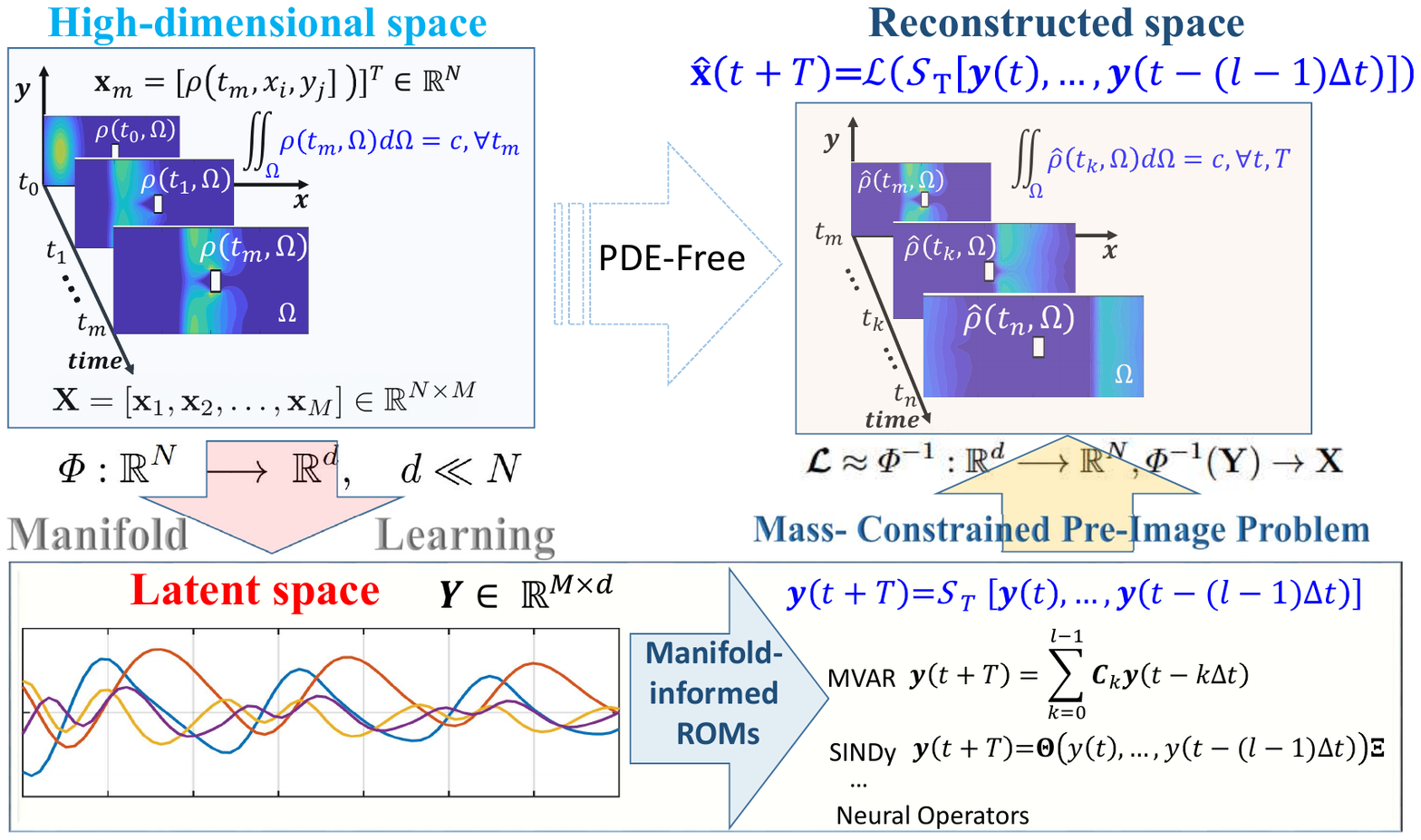}
    \caption{Schematic of the PDE-free mass-preserving approach. 
    (1) Manifold learning: projection of high-dimensional observations onto latent spaces; 
    (2) Learning manifold-informed ROMs (i.e., learning the solution operator in the latent space); 
    (3) Solution of the pre-image problem: construction of the map between the learned latent-space dynamics and the spatio-temporal dynamics (i.e., reconstruction of the solution operator of the unavailable PDE).}
\label{fig:overview_framework}
\end{figure}

A schematic of the proposed framework is shown in Fig.~\ref{fig:overview_framework}.
Assuming a given set of $M$ spatio-temporal snapshots of a field, say 
$\{\mathbf{x}_m\}_{m=1,\ldots, M}$, the aim is to learn a solution operator on a low-dimensional latent space using manifold-informed ROMs, and then to reconstruct—in a fully data-driven, \textit{Equation-free} manner—an approximation of the apparent solution operator in the high-dimensional input space.

Each snapshot at time $t_m$ may contain, for example, the density across the spatial domain; i.e., $\mathbf{x}_m = [\rho(t_m, x_i, y_j)]^\top \in \mathbb{R}^{N}$, where $i=1,\ldots,N_x$, $j=1,\ldots,N_y$, and $N=N_x \times N_y$ is the number of points in the computational grid. 
Let us represent the dataset as 
\begin{equation}
\label{eq:snapshot}
    \mathbf{X} = [\mathbf{x}_1, \mathbf{x}_2,\dots,\mathbf{x}_M] \in \mathbb{R}^{N \times M},
\end{equation}
with columns \(\mathbf{x}_k \in \mathbb{R}^{N}\) being the density fields. 
The key assumption of the methodology is that these data points lie on a smooth $d$-dimensional manifold \(\mathcal{M}\subset\mathbb{R}^N\), and we seek a data‐driven mapping
\begin{equation}
\varPhi:\mathbb{R}^N\;\longrightarrow\;\mathbb{R}^d,
\quad d\ll N, 
\label{eq:HD2LDmap}
\end{equation}
such that \(\mathbf{y}_m=\varPhi(\mathbf{x}_m)\in \mathbb{R}^d\) captures the dominant coherent structures and emergent dynamics in the latent space, where $d$ approximates the intrinsic dimension of $\mathcal{M}$. The construction of this mapping via manifold learning and in particular DIffusion Maps (POD is also used for comparison purposes) constitutes the first step of the proposed framework.

Once $\varPhi$ is obtained, one can learn surrogate ROMs of the high-dimensional dynamics embedded in the latent space, i.e., learn the solution operator $\mathcal{S}_T$ of the reduced-order dynamics as:
\begin{equation}
\boldsymbol{y}(t+T)=\mathcal{S}_{T}[\boldsymbol{y}(t),\dots,\boldsymbol{y}(t-(l-1)\Delta t)],
\end{equation}
via machine-learning surrogates, where $T$ is the time horizon of the prediction, $l$ is the number of delays, $\Delta t$ is the sampling interval of the original data. Here, we employed SINDy and standard linear MVAR models with delayed coordinates to both demonstrate the framework and provide an indicative comparison of the two approaches.

In the final step, we solve the so-called pre-image problem, i.e., we seek a `lifting operator' $\mathcal{L}$ to reconstruct “on demand” the apparent solution operator $\mathcal{S}^{X}_T$ of the high-dimensional spatio-temporal dynamics from the latent dynamics. This allows us to construct the map on the reconstructed space as:
\begin{equation}
\hat{\boldsymbol{x}}(t+T)=\mathcal{L}\big(\mathcal{S}_{T}[\boldsymbol{y}(t),\dots,\boldsymbol{y}(t-(l-1)\Delta t]\big)
:=\mathcal{S}^{X}_T=\mathcal{L}\circ \mathcal{S}_{T}, 
\end{equation}
thus bypassing the need, as in the Equation-free framework~\cite{kevrekidis2003equation}, to learn explicitly the high-dimensional map itself. Importantly, we theoretically show that the linear POD lifting operators, as well as the $k$-NN algorithm used to construct the lifting operator for DMs, conserve the mass in the input/high-dimensional space. This property is crucial in a wide range of complex systems governed by conservation laws—such as in crowd dynamics and CFD.

Our approach yields compact, computationally efficient surrogates that are non-intrusive (they do not require any explicit knowledge of the governing PDE or parts of it) and are well suited to the Equation-free workflow~\cite{kevrekidis2003equation}: fast prediction in latent space followed by data-driven reconstruction to recover high-fidelity solutions. 

\subsection{Manifold Learning and the Pre-image/Reconstruction Problem}
\label{subsec:pca}
In the first step, we determine the map $\varPhi$ from the high-dimensional space to the latent space in Eq.~\eqref{eq:HD2LDmap} via manifold learning. As discussed, here, for comparison purposes, we consider both POD, and Diffusion Maps (DMs), equipped with a $k-NN$ decoder. Within the POD context, one assumes a linear map $\varPhi(\boldsymbol{x}_m) = \mathbf{W}^\top\boldsymbol{x}_m$, with $\mathbf{W} \in \mathbb{R}^{N\times d}$ containing the leading POD modes computed from the covariance matrix of $\mathbf{X}$ over all $t_m$, yielding a projection that maximizes variance in the $L^2$ sense. Nonlinear manifold learning methods (e.g., kernel-PCA, Laplacian Eigenmaps~\cite{belkin2003laplacian,belkin2008towards}, Diffusion Maps~\cite{coifman2005geometric,coifman2006geometric,nadler2006diffusion,coifman2008diffusion,evangelou2022double,evangelou2022parameter}, and Autoencoders~\cite{vlachas2022multiscale,romor2023non}) generalize this approach by constructing a nonlinear map $\varPhi$ that preserves intrinsic geometric properties of the data. Here, we used both POD and Diffusion Maps. 

When using POD, the lifting task is trivial: the full state is reconstructed as a linear combination of POD basis modes weighted by the latent coefficients. In the case of nonlinear manifold learning algorithms, such as Diffusion Maps, we aim to learn an approximation of the inverse map (the lifting operator/decoder):
\begin{equation}
    \mathcal{L}
    \approx \varPhi^{-1}:\mathbb{R}^d\longrightarrow \mathbb{R}^N, \quad \varPhi^{-1}(\mathbf{Y}) \rightarrow \mathbf{X},
\end{equation}
for new samples on the manifold, say $\boldsymbol{y}^* \notin \varPhi(\mathbf{X})$, \textit{i.e.}, those produced by the MVAR models. This inverse problem is referred to as the ``out-of-sample extension pre-image'' problem, which is ill-posed. Compactly, this can be written as
\begin{equation}
\varPhi(\mathcal{L}(\mathbf{y}^*)) \approx \mathbf{y}^*.
\label{liftOP}
\end{equation}

\subsubsection{POD-based reconstruction conserves mass.}
\label{subsubsec:POD}

Considering a density field $\rho(t,x,y)$ as a paradigm, in POD, we decompose the fluctuations of the density field $\rho(t,x,y)$ with respect to the temporal mean $\overline{\rho}(x,y)$ as
\begin{equation}
\label{eq:POD_def}
\rho'(t,x, y) = \rho(t,x, y) - \overline{\rho}(x, y) = \sum_{i=1}^{\infty} y_i(t)\, \boldsymbol{\varphi}_i(x, y),
\end{equation}
where the spatial modes $\boldsymbol{\varphi}_i(x,y)$ are mutually orthogonal. The discrete modes $\boldsymbol{\varphi}_i$ are obtained via the method of snapshots~\cite{Sirovich}, namely by solving the eigenvalue problem for the covariance matrix $\mathbf{X}^\top\mathbf{X} \in \mathbb{R}^{M \times M}$:
\begin{equation} 
\label{eq:eigprobQQt}
\mathbf{X}^\top\mathbf{X} \boldsymbol{\psi}_i = \lambda_i \boldsymbol{\psi}_i, \quad i=1, \ldots, M,
\end{equation}
where $\boldsymbol{\varphi}_i  = \mathbf{X} \boldsymbol{\psi}_i / \sqrt{\lambda_i}$ is the $i$th mode, and $\lambda_i$, $\boldsymbol{\psi}_i$ are the corresponding eigenvalue and eigenvector of the covariance matrix, sorted in descending order such that $\lambda_1 > \ldots > \lambda_M$. Here, $\mathbf{X} \in \mathbb{R}^{N \times M}$ is the snapshot matrix. By retaining the leading $d \ll M$ modes, we construct the reduced-order POD basis parameterizing the manifold, which is subsequently used to project the high-dimensional density field onto the temporal coordinates $y_i(t)$, for $i=1, \dots, d$.

Using the POD basis, we define the mapping $\varPhi$ in Eq.~\eqref{eq:HD2LDmap} as $\mathbf{y} = \varPhi(\mathbf{x}) = \mathbf{W}^\top\mathbf{x}$, where $\mathbf{W} = [\boldsymbol{\varphi}_1,\ldots,\boldsymbol{\varphi}_d] \in \mathbb{R}^{N\times d}$, thus defining the latent variables as the leading POD coefficients $\mathbf{y} = [y_1,\ldots,y_d]^\top$. Note that the POD basis provides a parameterization of the manifold for both seen points $\mathbf{x}_m$ and unseen ones $\mathbf{x}^*_m$. Hence, it can be employed for any point in the high-dimensional space, providing the projection subsequently used for ROM construction via the MVAR model.

We now prove that the reconstructed solution $\hat{\mathbf{x}}$ via POD preserves the total mass in each snapshot.
\begin{prop}
\label{prop1}
Let the dataset be 
\begin{equation}
    \mathbf{X} = [\mathbf{x}_1, \mathbf{x}_2, \dots, \mathbf{x}_M] \in \mathbb{R}^{N \times M},
\end{equation}
with columns $\mathbf{x}_k \in \mathbb{R}^{N}$ representing normalized to sum to 1 density fields. If the total mass of the original data is preserved, i.e.,
\begin{align}
    \begin{bmatrix}
      \sum_{i=1}^{N} x_{i1} & \sum_{i=1}^{N} x_{i2} & \dots & \sum_{i=1}^{N} x_{iM}
    \end{bmatrix}
    = \mathbf{1}_M^\top,
\end{align}
then the mass of the \textit{reconstructed/decoded} field computed by POD as
\begin{equation}
\hat{\mathbf{X}} = \mathbf{U}_d 
\mathbf{U}_d^\top \tilde{\mathbf{X}} +  \mathbf{X} (\mathbf{I}_M - \mathbf{H}), \quad \mathbf{H} = \mathbf{I}_M - \frac{1}{M} \mathbf{1}_M \mathbf{1}_M^\top,
\label{eq:recon_theorem}
\end{equation}
where $\mathbf{U}_d = [\mathbf{u}_1,\mathbf{u}_2,\dots, \mathbf{u}_d] \in \mathbb{R}^{N \times d}$ is the orthonormal basis formed by the first $d$ left singular vectors of the SVD of the centered matrix $\tilde{\mathbf{X}} = \mathbf{X} \mathbf{H}$, is also preserved, i.e.,
\begin{align}
    \begin{bmatrix}
      \sum_{i=1}^{N} \hat{x}_{i1} & \sum_{i=1}^{N} \hat{x}_{i2} & \dots & \sum_{i=1}^{N} \hat{x}_{iM}
    \end{bmatrix}
    = \mathbf{1}_M^\top.
\end{align}
\end{prop}

\begin{proof}
Since the sum of each column of $\mathbf{X}$ is 1 by hypothesis, we have
\[
\mathbf{1}_N^\top \mathbf{X} = \mathbf{1}_M^\top.
\]
Then, for the centered matrix $\tilde{\mathbf{X}} = \mathbf{X} \mathbf{H}$, we obtain
\begin{align}
\mathbf{1}_N^\top \tilde{\mathbf{X}} = \mathbf{1}_N^\top \mathbf{X} \mathbf{H} = \mathbf{1}_M^\top \mathbf{H} = \mathbf{1}_M^\top - \frac{1}{M} \mathbf{1}_M^\top \mathbf{1}_M \mathbf{1}_M^\top = \mathbf{1}_M^\top - \mathbf{1}_M^\top = \mathbf{0}_M^\top.
\end{align}
This implies that $\mathbf{1}_{N}$ is orthogonal to the column space of $\tilde{\mathbf{X}}$, and therefore to the left singular vectors of its SVD, which provide an orthogonal basis for that column space. Thus,
\begin{equation}
    \mathbf{1}_{N}^\top \mathbf{U}_d = \boldsymbol{0}^\top_d,
\end{equation}
where $d=1,2,\dots, r$, and $r$ is the rank of $\mathbf{X}$. Multiplying Eq.~\eqref{eq:recon_theorem} by $\mathbf{1}_{N}^\top$ yields
\begin{equation}
\mathbf{1}_N^\top \hat{\mathbf{X}} = \mathbf{1}_N^\top \mathbf{U}_d \mathbf{U}_d^\top \tilde{\mathbf{X}} + \mathbf{1}_N^\top \mathbf{X} (\mathbf{I}_M - \mathbf{H})  = \mathbf{0}_M^\top + \frac{1}{M} \mathbf{1}_N^\top \mathbf{X} \mathbf{1}_M \mathbf{1}_M^\top = \mathbf{1}_M^\top,
\end{equation}
since $\mathbf{1}_M^\top \mathbf{1}_M = M$. We therefore obtain
\begin{align}
    \mathbf{1}_{N}^\top \hat{\mathbf{X}} =
    \begin{bmatrix}
    \sum_{i=1}^{N} \hat{x}_{i1} &  \sum_{i=1}^{N} \hat{x}_{i2} & \dots & \sum_{i=1}^{N} \hat{x}_{iM}  
    \end{bmatrix} = \mathbf{1}^\top_{M},
\end{align}
hence showing that the total mass in each snapshot is preserved under POD reconstruction.
\end{proof}

The above result follows from the well-known property of POD to preserve symmetries~\cite{aubry1993preserving}.
At this point, we note that mass conservation and pointwise positivity are distinct properties. While POD-based reconstruction preserves total mass, it does not guarantee non-negativity, and small negative values may appear despite positive input data. Since the mean is added back, such values typically occur in low-density regions. For visualization purposes only, one can set negative entries to zero and renormalize the field by the mass of its nonnegative part, thereby enforcing positivity while preserving total mass. This does not influence the training process, since learning is performed entirely in the latent space and the reconstruction is applied only a posteriori.

\subsubsection{Encoding with Diffusion Maps and Decoding with the k-NN algorithm}
\label{subsubsec:DMs}

The purpose of DMs is to construct a nonlinear mapping from a high-dimensional space to a low-dimensional subspace while preserving the intrinsic geometry of the underlying manifold. We follow the theoretical formulation and numerical implementation of DMs presented in earlier works \cite{dsilva2018parsimonious,holiday2019manifold,Patsatzis_2023,chin2024enabling,gallos2024data}. 

Assume that the data lie on a smooth, low-dimensional manifold $\mathcal{M} \subset \mathbb{R}^N$. Diffusion Maps then aim to obtain low-dimensional embeddings $\mathbf{y} \in \mathbb{R}^d$, with $d \ll N$, collected in the matrix $\mathbf{Y} \in \mathbb{R}^{M \times d}$, such that Euclidean distances between points $\mathbf{y}$ approximate the diffusion distances between the original points \cite{nadler2006diffusion}.

The algorithm begins by defining a similarity measure between pairs of data points $\boldsymbol{x}_i, \boldsymbol{x}_j \in \mathbf{X}$, $\forall i,j=1,\ldots,M$, in the high-dimensional space. Using the Euclidean norm $d_{ij} = \|\boldsymbol{x}_i - \boldsymbol{x}_j\|$, we construct a Gaussian kernel $k(\boldsymbol{x}_i, \boldsymbol{x}_j)$, which defines the affinity matrix:
\begin{equation}
	\mathbf{A} = [a_{ij}] = [k(\boldsymbol{x}_i, \boldsymbol{x}_j)] = \exp\left(-\frac{\|\boldsymbol{x}_i - \boldsymbol{x}_j\|^2}{\epsilon^2}\right),
	\label{eq:affinity_matrix}
\end{equation}
where $\epsilon$ controls the local neighborhood size in the high-dimensional space. In our implementation, we set $\epsilon = \mathrm{median}(d_{ij})$, which promotes a relatively large neighborhood. Other strategies for selecting $\epsilon$ exist \cite{singer2009detecting,gallos2021construction}.

Next, the $M \times M$ Markov transition matrix $\mathbf{M}$ is formed by row-normalizing the affinity matrix:
\begin{equation}
	\mathbf{M} = \mathbf{D}^{-1} \mathbf{A}, \quad \text{with} \quad \mathbf{D} = \operatorname{diag}\left(\sum_{j=1}^{M} a_{ij}\right).
	\label{eq:Markovian_matrix}
\end{equation}
Each entry $\mu_{ij}$ of $\mathbf{M}$ represents the probability of moving from point $i$ to point $j$ in the high-dimensional space:
\begin{equation}
	\mu_{ij} = \operatorname{Prob}\left(X_{t+1} = \boldsymbol{x}_j \mid X_t = \boldsymbol{x}_i\right).
\end{equation}
Equivalently, using the kernel,
\begin{equation}
	\mu_{ij} = \frac{k(\boldsymbol{x}_i, \boldsymbol{x}_j)}{\operatorname{deg}(\boldsymbol{x}_i)}, \quad \text{with} \quad \operatorname{deg}(\boldsymbol{x}_i) = \sum_{j=1}^{M} k(\boldsymbol{x}_i, \boldsymbol{x}_j),
\end{equation}
recovering Eq.~\eqref{eq:Markovian_matrix}.

The transition matrix $\mathbf{M}$ is similar to the symmetric, positive-definite matrix $\hat{\mathbf{M}} = \mathbf{D}^{-1/2} \mathbf{A} \mathbf{D}^{-1/2}$, which allows an eigendecomposition
\begin{equation}
	\mathbf{M} = \sum_{i=1}^{M} \lambda_i \mathbf{w}_i \mathbf{u}_i^\top,
\end{equation}
where $\lambda_i \in \mathbb{R}$ are eigenvalues and $\mathbf{w}_i, \mathbf{u}_i \in \mathbb{R}^M$ are left and right eigenvectors, satisfying $\langle \mathbf{w}_i, \mathbf{u}_j \rangle = \delta_i^j$. The right eigenvectors $\mathbf{u}_i$ span an orthonormal basis for the low-dimensional subspace $\mathbb{R}^d$, and the best $d$-dimensional approximation is obtained from the $d$ largest eigenvalues.

The standard DMs embedding maps each snapshot $\boldsymbol{x}_m$ to
\begin{equation}
	\mathbf{y}_m = (\lambda_1 u_{1,m}, \ldots, \lambda_d u_{d,m}), \quad m = 1, \ldots, M,
\end{equation}
where $u_{i,m}$ denotes the $m$-th component of the $i$-th right eigenvector corresponding to the $i$-th largest non-trivial eigenvalue $\lambda_i$. This embedding approximates the diffusion distance in the high-dimensional space by Euclidean distance in the embedded space:
\begin{equation}
	D_t^2(\boldsymbol{x}_i, \boldsymbol{x}_j) = \left\| \mu_t(\boldsymbol{x}_i, \cdot) - \mu_t(\boldsymbol{x}_j, \cdot) \right\|_{L_2, 1/\operatorname{deg}}^2 = \sum_{k=1}^{M} \frac{\left(\mu_t(\boldsymbol{x}_i, \boldsymbol{x}_k) - \mu_t(\boldsymbol{x}_j, \boldsymbol{x}_k)\right)^2}{\operatorname{deg}(\boldsymbol{x}_k)},
\end{equation}
with $\mu_t(\mathbf{x}_i, \cdot)$ the $i$-th row of $\mathbf{M}^t$. In our computations, we use $t = 1$.

In practice, the embedded dimension $d$ is determined by the
spectral gap of the eigenvalue ratio of the transition matrix $\textbf{M}$,
assuming that the first $d$ leading eigenvalues are adequate to provide a good approximation of the diffusion distance between all pairs of points.
The resulting DMs embedding is constructed from the retained eigenpairs $\{\lambda_i, \mathbf{u}_i\}_{i=1}^d$. The restriction operator (encoder) $\varPhi$, evaluated on a data point $\mathbf{x}_m$, is
\begin{equation}
	\varPhi(\mathbf{x}_m) = (\lambda_1 u_{1,m}, \ldots, \lambda_d u_{d,m}) = \mathbf{y}_m \in \mathbb{R}^d, \quad m = 1, \ldots, M.
\end{equation}
For new, unseen points, we employ the Nyström method \cite{nystrom1929uber,coifman2006geometric,chiavazzo2014reduced,evangelou2022double,Patsatzis_2023}.

\paragraph{Decoding with $k$-NN}
Here, for the $k$-NN approach for the solution for the pre-image problem \cite{chin2024enabling}, we prove the following proposition.
\begin{prop}
Let the assumptions for the data matrix $\mathbf{X}$ in Proposition~\ref{prop1} hold, i.e., columns $\mathbf{x}_k \in \mathbb{R}^N$ are normalized to sum to 1. Then the solution of the pre-image problem using the k-NN algorithm with convex interpolation preserves both the total mass and positivity of the reconstructed state.
\end{prop}

\begin{proof}
Construct the lifting operator (decoder) $\mathcal{L}$ using k-NN as (see in \cite{chin2024enabling}):
\begin{equation}
    \mathbf{x}^* = \mathcal{L}(\mathbf{y}^*) = \sum_{k=1}^K b_k \mathbf{x}_{S(k)},
\end{equation}
where $\mathbf{x}_{S(k)} \in \mathbf{X}$ are the nearest neighbors in the high-dimensional space corresponding to $\mathbf{y}^*$ in latent space. The convex weights $b_k$ satisfy $\sum_{k=1}^K b_k = 1$ and $b_k \in [0,1]$.  

Since $\mathbf{x}^*$ is a convex combination of normalized columns of $\mathbf{X}$, we have
\begin{equation}
    \mathbf{1}^\top \mathbf{x}^* = \sum_{k=1}^K b_k \mathbf{1}^\top \mathbf{x}_{S(k)} = \sum_{k=1}^K b_k = 1.
\end{equation}
Hence, the reconstructed solution preserves the total mass. Furthermore, a weighted sum of this form—nonnegative weights that sum to one—produces a point $\mathbf{x}^*$ that lies in the convex hull of its neighbors and is positive if its neighbors are positive.
\end{proof}

\subsection{Manifold-Informed ROMs}
Manifold-informed ROMs (e.g., using multivariate autoregressive models (MVARs), neural networks, Gaussian Processes, SINDy, or NOs) can be used without changing the proposed framework, as shown in following Section~\ref{subsec:SINDy}.
Here, we employed two approaches: vanilla linear MVARs with multiple delayed coordinates and Sparse Identification of Nonlinear Dynamics (SINDy) \cite{brunton2016discovering,Rudy2017} with a single delay. For both approaches, we optimized embedding parameters and model hyperparameters (e.g., model order for MVAR,  sparsity threshold and library for SINDy) on the training sets, and selected the SINDy candidate library via cross-validation to balance predictive accuracy and model parsimony. More details are given in \ref{app:results_insights} and \ref{app:results_insights_fluid}.

\subsubsection{Manifold-Informed linear MVARs}
\label{subsec:multivariate}

As shown in \cite{papaioannou2022time}, MVARs in delay-coordinate latent spaces \cite{shvartsman1998low,dylewsky2022principal,kemeth2022learning,axaas2023model,gallos2024data,patil2025separation} can outperform nonlinear machine learning surrogates when the latent manifold is well parameterized and the dynamics evolve around a single nonlinear basin of attraction. In this regime, the latent dynamics are approximately linearizable, and MVAR models capture the evolution with high fidelity. Training MVARs requires only linear regression, offering a computationally efficient alternative to deep nonlinear models such as neural networks, LSTMs, or GRUs. Furthermore, the learned MVAR coefficients provide interpretability of the mode interactions, and the framework naturally supports uncertainty quantification. 

For one-step-ahead prediction $(T=\Delta t)$, an MVAR model of order $l$ reads:
\begin{equation}
\mathbf{y}(t + \Delta t) = \sum_{k=0}^{l-1} \mathbf{C}_k \, \mathbf{y}(t - k\Delta t),
\label{eq:mar_latent_y}
\end{equation}
where $\mathbf{C}_k \in \mathbb{R}^{d \times d}$ are the autoregressive coefficient matrices at lag $k$. In compact form,
\begin{equation}
\mathbf{y}(t + \Delta t) = \mathbf{C} \, \mathbf{r}(t),
\label{eq:mar_compact_y}
\end{equation}
with $\mathbf{C} \in \mathbb{R}^{d \times (d \cdot l)}$ concatenating the $\mathbf{C}_k$, and the regressor vector
\begin{equation}
\label{eq:regressor}
\mathbf{r} = \begin{bmatrix} 
\mathbf{y}(t)^\top & \mathbf{y}(t - \Delta t)^\top & \cdots & \mathbf{y}(t - (l-1)\Delta t)^\top
\end{bmatrix}^\top \in \mathbb{R}^{d \cdot l}.
\end{equation}

To estimate the MVAR parameters, we minimize the total squared prediction error over all sequences corresponding to different initial conditions. Let $\{\mathbf{y}^{(n)}(t)\}_{t=0}^{\tau}$ denote the latent sequence for the $n$-th initial condition, $n = 1,\dots, N_{ic}$, and $\tau = t_f - t_0$ the sequence length. The corresponding regressor is
\[
\mathbf{r}^{(n)} = \begin{bmatrix} 
\mathbf{y}^{(n)}(t)^\top & \mathbf{y}^{(n)}(t - \Delta t)^\top & \cdots & \mathbf{y}^{(n)}(t - (l-1)\Delta t)^\top
\end{bmatrix}^\top.
\]

The Ordinary Least Squares (OLS) solution \cite{Lutkepohl2005} minimizes
\begin{equation}
\min_{\mathbf{C}} \sum_{n = 1}^{N_{ic}} \sum_{t = (l-1)\Delta t}^{\tau - \Delta t} \left\| \mathbf{y}^{(n)}(t + \Delta t) - \mathbf{C} \, \mathbf{r}^{(n)} \right\|_2^2,
\label{eq:mar_optimization_latent_multi_y}
\end{equation}
yielding the closed-form solution
\begin{equation}
\mathbf{C}^\top = \left( \mathbf{R}^\top \mathbf{R} \right)^{-1} \mathbf{R}^\top \mathbf{Y},
\label{eq:ols_solution_y}
\end{equation}
where $\mathbf{R} \in \mathbb{R}^{M \times (d \cdot l)}$ stacks all regressor vectors and $\mathbf{Y} \in \mathbb{R}^{M \times d}$ contains the target latent states. Temporal stationarity of the latent states was verified via the Augmented Dickey–Fuller (ADF) test \cite{DickeyFuller1979}.

The optimal lag order $l$ is selected using the Bayesian Information Criterion (BIC) \cite{Lutkepohl2005}:
\begin{equation}
\text{BIC}(l) = -2 \cdot \log\!\left( \hat{\mathcal{L}}_h(l) \right) + \log(N_\tau) \cdot T_c,
\label{eq:bic_y}
\end{equation}
where $\hat{\mathcal{L}}_h(l)$ is the maximum likelihood for lag $l$ and $T_c = l \cdot d^2$ is the number of autoregressive coefficients. The optimal lag minimizes $\text{BIC}(l)$. Once trained, the MVAR model provides a compact, interpretable representation of the latent temporal dynamics, suitable for forecasting, simulation, and uncertainty quantification.

\subsubsection{Manifold-Informed SINDy}
\label{subsec:SINDy}

SINDy seeks a parsimonious description of the governing dynamics by representing the time evolution of the latent variables as a sparse linear combination of candidate nonlinear functions drawn from a predefined library \cite{brunton2016discovering,Champion_SINDy}. Within the present framework, SINDy is applied directly in the learned latent space, thereby avoiding the curse of dimensionality associated with discovering nonlinear PDEs in the original high-dimensional space (see also the discussion in \cite{Rudy2017}).

Let $\mathbf{y}(t) \in \mathbb{R}^d$ denote the latent coordinates obtained via POD or DMs. Within the discrete map context, for one-step-ahead predictions, SINDy seeks to approximate the latent dynamics as
\begin{equation}
\mathbf{y}(t+\Delta t) = \Theta(\mathbf{r}) \, \Xi.
\label{eq:sindy_model}
\end{equation}
where $\Theta(\mathbf{r}) \in \mathbb{R}^{p}$ is a library of $p$ candidate nonlinear functions and $\Xi \in \mathbb{R}^{p \times d}$ is a sparse coefficient matrix whose nonzero entries identify the active terms governing the dynamics. This continuous-time formulation provides conceptual motivation for the method.

Typical libraries include polynomial functions up to a prescribed order, Fourier basis functions, or combinations thereof, optionally augmented with cross terms. In the numerical experiments performed in this work, we consider polynomial libraries of order one, two and three, denoted by $p(1)$, $p(2)$ and $p(3)$, respectively, as well as Fourier libraries, denoted by $f$. When explicitly excluding cross-coupling polynomial terms, we use the notation $p_{\mathrm{nc}}(\cdot)$. When a Fourier library is employed, $\Theta(\cdot)$ consists of sine and cosine functions of the latent coordinates and their delays, enabling the identification of oscillatory and advective structures in the latent dynamics.

The sparse coefficient matrix $\Xi$ is identified via sparse regression using sequentially thresholded least squares or equivalent $\ell_1$-regularized optimization procedures \cite{brunton2016discovering}. As in the MVAR case, SINDy operates entirely in the latent space and is therefore fully compatible with the proposed PDE-free, manifold-informed framework. Both families of ROMs are systematically evaluated and compared in Section~\ref{sec:results}.

\subsection{Assessing the Reconstruction Error}
\label{subsec:recon_err}

Once latent coordinates are obtained (via POD or DMs), a high-dimensional state $\rho \in \mathbb{R}^{N}$ can be projected into the latent space and then lifted back to the high-dimensional space to obtain a reconstructed field $\hat{\rho} \in \mathbb{R}^{N}$. The reconstruction quality is evaluated against the ground-truth $\rho$ using the absolute and relative \(\ell_2\) errors:
\begin{equation}
\varepsilon_2(t) = \| \rho(t,\cdot,\cdot) - \hat{\rho}(t,\cdot,\cdot) \|_2, 
\qquad
\varepsilon^r_2(t) = \frac{\| \rho(t,\cdot,\cdot) - \hat{\rho}(t,\cdot,\cdot) \|_2}{\| \rho(t,\cdot,\cdot) \|_2},
\label{eq:L2_errors}
\end{equation}
as well as the \emph{Wasserstein-$1$ distance} ($W_1$) \cite{villani2009optimal,piccoli2014generalized,cristiani2014evolution,piccoli2016properties} defined as
\begin{equation}
W_1(t) =  \inf_{\gamma \in \Gamma(\rho, \hat{\rho})} \int_{\Omega \times \Omega} \|x - y\|_2 \, d\gamma(x,y),
\label{eq:wasserstein}
\end{equation}
where $\Gamma(\rho, \hat{\rho})$ is the set of joint distributions (couplings) with marginals $\rho$ and $\hat{\rho}$, and $\|x-y\|_2$ is the Euclidean distance.  

The Wasserstein distance is particularly suitable for spatio-temporal dynamics as it accounts for both magnitude and spatial rearrangements: small translations of a density pattern yield small $W_1$, unlike \(\ell_2\) or \(\ell_\infty\) norms that can overestimate error due to local misalignments. Here, the computation of the Wasserstein distances is performed in MATLAB using a simplified, discrete approach by summing absolute differences between sorted samples ~\cite{wesserstein}.

Here, based on the distributions of $\varepsilon_2(t)$,  $\varepsilon^r_2(t)$, $W_1(t)$,  across all snapshots and initial conditions, we report the mean values and the 10$^{\text{th}}$–90$^{\text{th}}$ percentiles. 

\section{Case Studies}
\label{sec:casestudies}

For our illustration, we considered two benchmark problems with mass conservation: (a) The Hughes model of crowd dynamics describing individuals minimizing travel time, formalized as a coupled PDE system in which the density evolves according to a continuity equation, while motion follows the gradient of an eikonal potential that naturally guides agents around obstacles and high-density regions; in the present case study, we considered the Hughes model applied to pedestrian flow in a corridor with an obstacle (see Fig.~\ref{fig:domain}(a));  (b) a fluid dynamics problem, which involves the spatio-temporal evolution of a passive tracer advected by a Navier–Stokes velocity generated by a three-cylinder (fluidic pinball) configuration (see Fig.~\ref{fig:domain}(b)).

\subsection{Hughes Model for Crowd Dynamics}
\label{sec:layout}

\begin{figure}
    \centering

    \begin{subfigure}{0.48\textwidth}
        \centering
        \includegraphics[width=\linewidth]{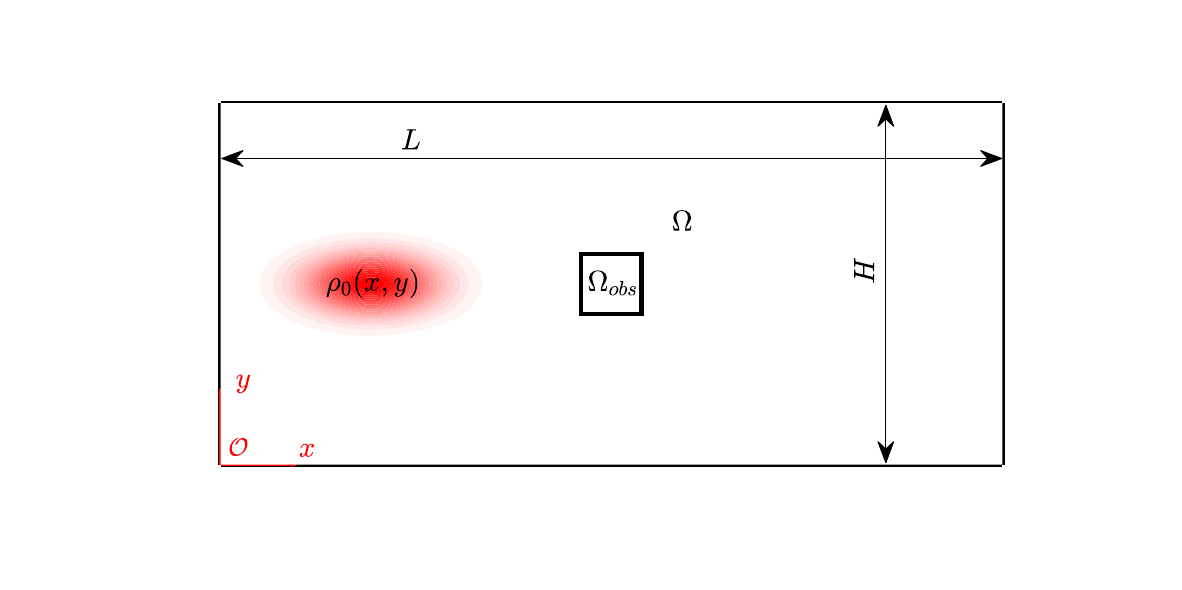}
        \caption{}
        \label{fig:domain_a}
    \end{subfigure}
    \hfill
    \begin{subfigure}{0.48\textwidth}
        \centering
        \includegraphics[width=\linewidth]{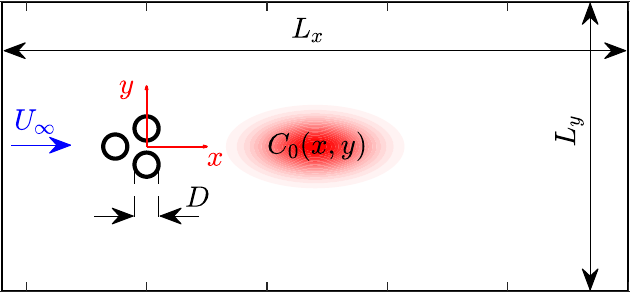}
        \caption{}
        \label{fig:domain_b}
    \end{subfigure}
    \caption{Schematic representation of the case studies: (a) pedestrian motion in a corridor with an obstacle is modeled using the Hughes model; (b) a fluid dynamics setup, where a passive tracer is transported by the Navier–Stokes velocity field produced by a three-cylinder “fluidic pinball” arrangement.}
    \label{fig:domain}
\end{figure}

A seminal model for approximating behavioral crowd dynamics via PDEs is the continuum Hughes model~\cite{hughes2002continuum,goatin2013wave,amadori2014existence,amadori2023mathematical}, which combines a conservation law for the crowd density with an Eikonal equation to compute shortest-time paths to a common goal. Crucially, the \textit{optimal} path selection is assumed to depend on the global distribution of pedestrians in the environment, necessitating a dynamic resolution of the Eikonal equation at every time step. Although the local density–velocity relationship appears explicitly in the conservation law, the coupling with the Eikonal equation imparts a global character to the dynamics. This allows the model to reproduce complex behavioral features such as obstacle avoidance and re-routing in congested environments~\cite{coscia_canavesio}. Both theoretical investigations and empirical studies support the notion that pedestrian speed and direction are influenced by the full spatial distribution of the crowd, rather than just local surroundings~\cite{cristiani2014multiscale}. Consequently, while such global dependencies enhance model fidelity, they also introduce substantial challenges in terms of analytical tractability and computational cost. This motivates the need for efficient and compact representations of crowd motion—specifically, ROMs accounting for the \textit{unknown hidden behavioral} dynamics~\cite{bellomo2008modelling,bellomo2023human,cristiani2014multiscale,kemeth2018emergent,bellomo2023behavioral,bellomo2023human}—so as to retain the essential features required to reproduce its global evolution.

The Hughes model~\cite{hughes2002continuum}  relies on three behavioral hypotheses: (i) pedestrian speed $f(\rho)$ depends solely on the local density $\rho$, following empirical pedestrian-flow relationships~\cite{greenshields1934study,lighthill1955kinematic}; (ii) pedestrians move in the direction that minimizes a potential function $\phi$, representing perceived travel effort or time; and (iii) individuals aim to minimize travel time while avoiding high-density regions, modeled via a separable cost function combining speed $f(\rho)$ and a discomfort factor $g(\rho)$. This yields a coupled PDE system for crowd density $\rho$ and potential $\phi$:
\begin{subequations}
\begin{align}
\frac{\partial \rho}{\partial t} &=\frac{\partial}{\partial x} \left( \rho f(\rho) \frac{\partial \phi}{\partial x} \frac{1}{\|\nabla \phi\|} \right) 
+ \frac{\partial}{\partial y} \left( \rho f(\rho) \frac{\partial \phi}{\partial y} \frac{1}{\|\nabla \phi\|} \right),\label{hughes_a}\\
\|\nabla \phi\| &=\frac{1}{f(\rho) g(\rho)},\label{hughes_b}
\end{align}
\end{subequations}
with $\|\nabla \phi\| = \sqrt{(\partial \phi / \partial x)^2 + (\partial \phi / \partial y)^2}$. Here, $f(\rho)$ represents the pedestrian speed, while $g(\rho)$ accounts for the pedestrian discomfort at high crowd densities. A linear velocity–density relationship is considered, expressed as
\begin{equation}
f(\rho) = v_f \left(1 - \frac{\rho}{\rho_m}\right),
\label{veldens}
\end{equation}
where $v_f$ denotes the free-flow speed and $\rho_m$ the maximum pedestrian density. In this study, these parameters are set to $v_f = 1~\text{m/s}$ and $\rho_m = 5~\text{people/m}^2$, respectively, and the discomfort factor is $g(\rho) = 1$. Additional details on data generation, as well as initial and boundary conditions, are provided in \ref{app:numerical_scheme}.

\begin{figure}
\centering
\includegraphics[width=0.9\textwidth]{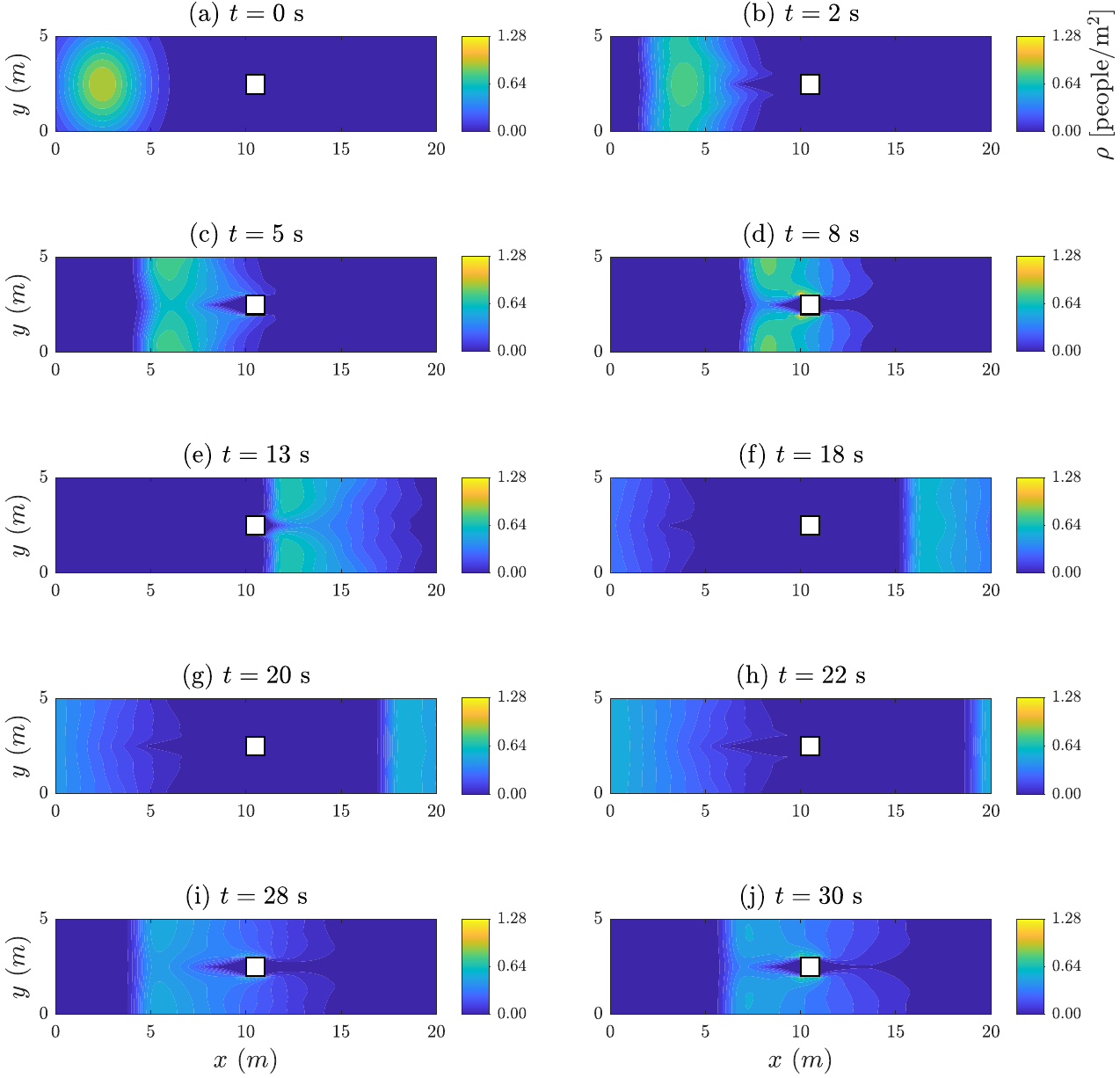}
\caption{Contour maps of the Hughes PDE solution $\rho(t,x,y)$ for the Gaussian initial condition $\rho_0(x,y)$ (see Eq.~\eqref{eq:initial_condition} in \ref{app:numerical_scheme}) with $x_0=2.5$~m, $y_0=2.5$~m, $\sigma_x=1.8$~m, $\sigma_y=1.8$~m, at different time instants: (a) $t=0$~s; (b) $t=2$~s; (c) $t=5$~s; (d) $t=8$~s; (e) $t=13$~s; (f) $t=18$~s; (g) $t=20$~s; (h) $t=22$~s; (i) $t=28$~s; (j) $t=30$~s.}
\label{fig:contour}
\end{figure}

A typical evolution of the density field is shown in Fig.~\ref{fig:contour}.
Initially, the crowd density (see Eq.~\eqref{eq:initial_condition} in \ref{app:numerical_scheme}) exhibits dominant streamwise motion (Figs.~\ref{fig:contour} (a),(b)), driven by the collective intent to reach the exit. Interestingly, the initial splitting of the crowd occurs before physically encountering the obstacle. This reflects the tendency of pedestrians to adjust trajectories in order to minimize perceived travel cost. As the crowd approaches the obstacle, the density bifurcates into two nearly symmetric (due to the given initial condition) streams (Figs.~\ref{fig:contour}(c),(d)), corresponding to two primary flows circumventing the obstruction. After passing the obstacle (Fig.~\ref{fig:contour}(e)), the two main streams gradually merge  (Fig.~\ref{fig:contour}(f)), reestablishing a unified front-like wave, and so on (Figs.~\ref{fig:contour}(g),(h),(i),(j)). Periodic boundaries reintroduce the crowd at the inlet, generating an almost cyclic pattern of inflow, splitting, and merging.

\subsection{Passive Tracer Transport in Navier–Stokes Flow}
\label{sec:pinball}

\begin{figure}
	\centering
	\includegraphics[scale=0.8]{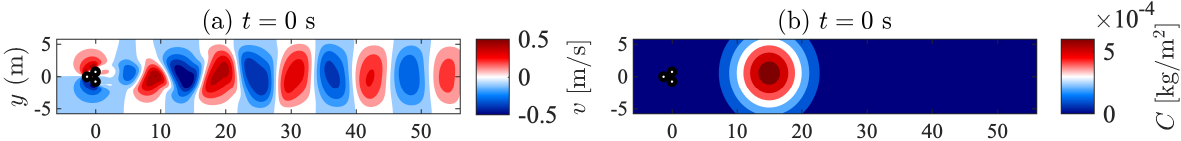}
    \vspace{0.1cm}
    	\includegraphics[scale=0.8]{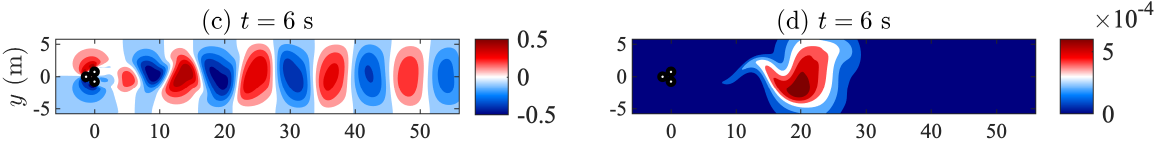}
            \vspace{0.1cm}
        	\includegraphics[scale=0.8]{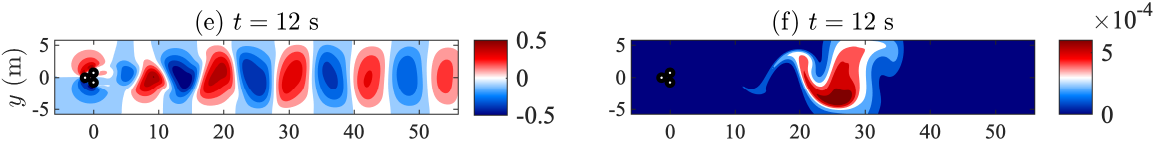}
                \vspace{0.1cm}
            	\includegraphics[scale=0.8]{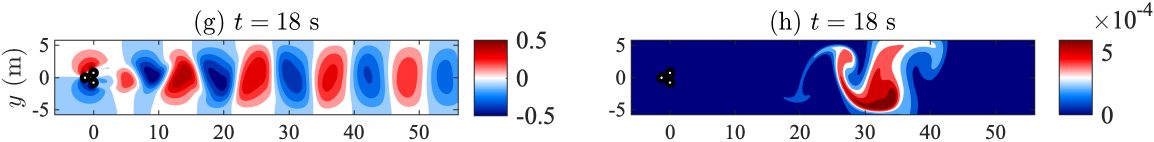}
                    \vspace{0.1cm}
                	\includegraphics[scale=0.8]{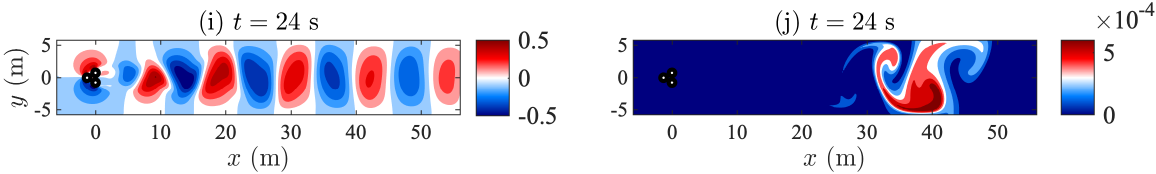}
	\caption{\label{fig:pinball_contour} Contour maps of the Navier-Stokes cross-stream velocity component $v(t,x,y)$ (left panels) and the passive tracer distribution $C(t,x,y)$ corresponding to the Gaussian initial condition $C_0(x,y)$ (see Eq.~\eqref{eq:initial_condition_pinball} in \ref{app:Navier-Stokes}) with $x_0=14.83$~m, $y_0=0.44$~m, $\sigma_x=4$~m, $\sigma_y=4.3$~m, at different time instants: (a)-(b) $t=0$~s; (c)-(d) $t=6$~s; (e)-(f) $t=12$~s; (g)-(h) $t=18$~s; (i)-(j) $t=24$~s.}
\end{figure}

The fluidic pinball is a flow configuration consisting of three rotatable cylinders of equal diameter $D$, whose axes are located at the vertices of an equilateral triangle, as sketched in Fig.~\ref{fig:domain}(b). The triangle has a centre-to-centre side length equal to $1.5D$ and is immersed in a viscous, incompressible flow with uniform upstream velocity $U_\infty$. The governing parameter of this configuration is the Reynolds number,
\begin{equation}
Re = \dfrac{U_\infty D}{\nu},
\end{equation}
where $\nu$ denotes the fluid kinematic viscosity. As $Re$ increases, the flow undergoes a sequence of bifurcations—including Andronov--Hopf, pitchfork, and Neimark--Sacker bifurcations—eventually leading to chaotic dynamics \cite{Deng_Noack_2020}. In this work, we focus on the regime at $Re = 30$, where the dynamics are characterized by spatio-temporal periodic oscillations.

The dataset is obtained from direct numerical simulations of the incompressible two-dimensional Navier--Stokes equations coupled with a linear advection equation for a passive scalar tracer,
\begin{subequations}
	\begin{eqnarray}
	\dfrac{\partial u}{\partial x} + \dfrac{\partial v}{\partial y} &=& 0, \label{eq:continuity} \\
	\dfrac{\partial u}{\partial t} + u \dfrac{\partial u}{\partial x} + v \dfrac{\partial u}{\partial y} &=& 
	-\dfrac{1}{\rho}\dfrac{\partial p}{\partial x} + \nu\left(\dfrac{\partial^2 u}{\partial x^2} + \dfrac{\partial^2 u}{\partial y^2}\right), \label{eq:momentum_u} \\
	\dfrac{\partial v}{\partial t} + u \dfrac{\partial v}{\partial x} + v \dfrac{\partial v}{\partial y} &=& 
	-\dfrac{1}{\rho}\dfrac{\partial p}{\partial y} + \nu\left(\dfrac{\partial^2 v}{\partial x^2} + \dfrac{\partial^2 v}{\partial y^2}\right), \label{eq:momentum_v} \\
	\dfrac{\partial C}{\partial t} + u \dfrac{\partial C}{\partial x} + v \dfrac{\partial C}{\partial y} &=& 0. \label{eq:advection_f}
	\end{eqnarray}
\end{subequations}

Here, $u$ and $v$ denote the streamwise ($x$) and cross-stream ($y$) velocity components, respectively, $p$ is the pressure, $\rho$ is the (constant) density, and $C$ is the passive tracer concentration. Additional details on data generation, as well as initial and boundary conditions, are provided in \ref{app:Navier-Stokes}.

In the Navier–Stokes system coupled with a passive tracer, the tracer concentration alone does not provide full observability of the underlying dynamics. The velocity field and pressure naturally act as hidden variables, mediating the transport, mixing, and deformation processes that govern the tracer evolution. As a result, the tracer dynamics effectively evolve on a latent manifold shaped  by these unobserved/hidden flow variables.

In Fig.~\ref{fig:pinball_contour}, typical evolutions of the cross-stream velocity field \(v\) (left panels) and of the passive tracer density \(C\) (right panels) are reported.
Note that the time \(t=0\) corresponds to the instant at which the tracer is injected into the domain (see Eq.~\eqref{eq:initial_condition_pinball} in \ref{app:Navier-Stokes}). The velocity field exhibits global spatio-temporal oscillations associated with the periodic vortex shedding that develops downstream of the pinball configuration in this Reynolds number regime (left panels in Fig.~\ref{fig:pinball_contour}). These oscillations arise from the \emph{Andronov--Hopf bifurcation} undergone by the pinball flow at \(Re \approx 18\), beyond which the steady base flow loses stability and a time-periodic limit cycle emerges \cite{Deng_Noack_2020}. The resulting von K\'arm\'an--type vortex street dominates the wake dynamics. As a consequence, the passive tracer is advected by the underlying velocity field and progressively deforms, stretches, and mixes as it is entrained by the coherent vortical structures. In particular, the tracer distribution closely follows the alternating vortex-shedding pattern, with regions of high tracer concentration being wrapped around the vortices and transported downstream. This process leads to an increasingly filamented tracer field whose spatial organization mirrors that of the underlying vorticity field, highlighting the strong coupling between wake dynamics and scalar transport in this regime.

\section{Numerical Results}
\label{sec:results}

\begin{figure}
    \centering
    \includegraphics[width=0.8\textwidth]{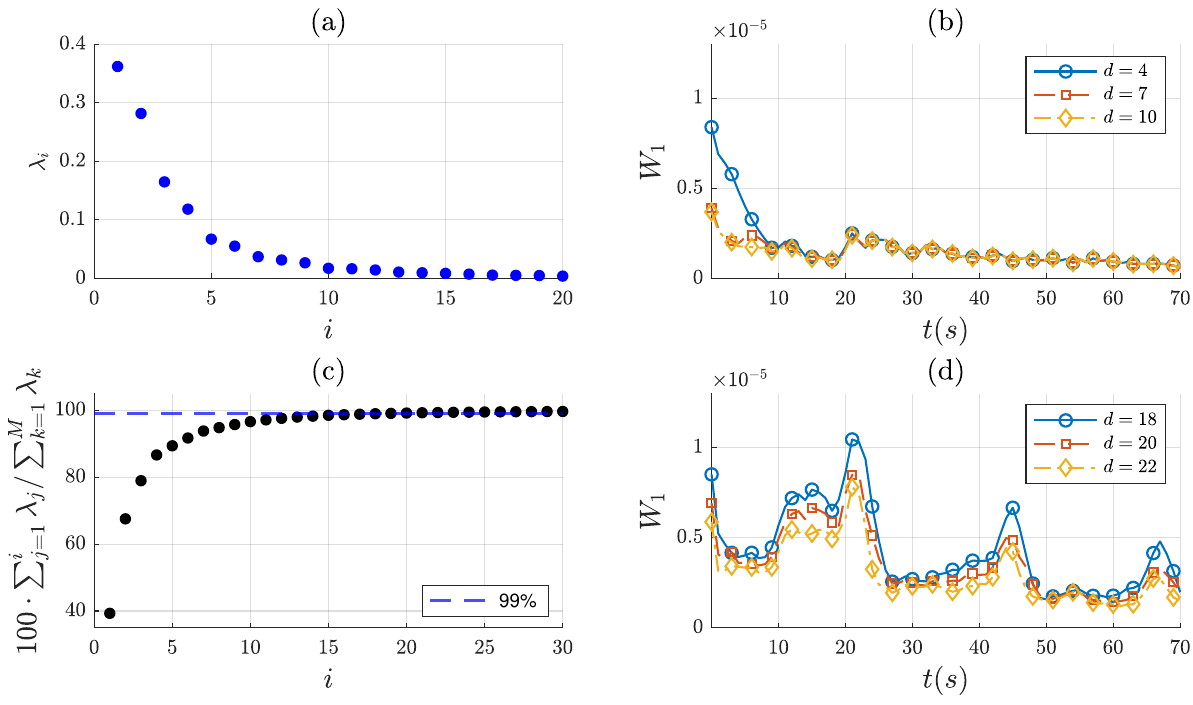}
    \caption{Crowd dynamics paradigm: (a) DMs eigenvalues \(\lambda_i\); (b) baseline mean reconstruction error \(W_1\) computed over \(d=4\), \(7\), and \(10\) DMs coordinates for the \(N_{\text{train}} = 40\) training simulations; (c) cumulative distribution of POD eigenvalues (black circles) and 99\% cumulative variance threshold (blue dashed line); (d) baseline mean reconstruction error \(W_1\) for \(d=18\), \(20\), and \(22\) POD coordinates on the training set.}
    \label{fig:PCA}
\end{figure}

\subsection{The Crowd Dynamics Paradigm}
\label{subsec:crowd}
Here, we evaluate the proposed methodology using numerical simulations of the Hughes PDE model for pedestrian dynamics. For our illustration, we have created \(N_{ic} = 100\) distinct initial conditions by uniformly sampling the parameters of the 2D Gaussian distribution (see Eq.~\eqref{eq:initial_condition} in \ref{app:numerical_scheme}) within the ranges \(x_0, y_0 \in [1.5, 3.5]\) and \(\sigma_x, \sigma_y \in [1.6,2]\), and integrate each simulation until \(t_f = 70\)~s. From these, we constructed three datasets: a training set \(X_{\text{train}}\) containing \(N_{\text{train}} = 40\) randomly selected simulations, a validation set with \(N_{\text{val}} = 20\) samples for model selection, and a test set \(X_{\text{test}}\) including the remaining \(N_{\text{test}} = 40\) simulations for out-of-sample evaluation.
The construction of \(X_{\text{train}}\), \(X_{\text{val}}\), and \(X_{\text{test}}\)—including parameter ranges and temporal resolution—follows the protocol detailed in Section~\ref{sec:layout}. 

In the first stage, as previously discussed, we employed both POD and DMs to learn a low-dimensional parametrization of the manifold. Given the high dimensionality of each snapshot (\(N = N_x \times N_y = 200 \times 50\)) and the large ensemble of training data (\(N_{\text{train}} \times N_t = 40 \times 700\)), performing manifold learning on the full dataset would be computationally prohibitive. Moreover, in the present crowd dynamics configuration, there exists a single nonlinear basin of attraction: trajectories originating from different initial conditions converge toward similar long-term behavior, with most variability concentrated in the initial transient phase and around the obstacle. During this transient—when eikonal-driven optimal paths bifurcate and spatial patterns (e.g., splitting flows and front-wave formations) emerge—the dynamics are particularly informative and of primary interest (see Fig.~\ref{fig:contour} and related discussion in Section~\ref{sec:layout}). 
Consequently, we learn the restriction and lifting operators from a subset of snapshots emphasizing the transient dynamics: two-thirds of the selected snapshots are drawn from the first \(10\,\text{s}\) of each simulation, and one-third from the remaining time. In total, \(X_{\text{train}}\) contains \(6{,}000\) snapshots for operator construction, with approximately \(4{,}000\) from the initial \(10\,\text{s}\) and \(2{,}000\) from the later interval. This design keeps the computational cost manageable while enhancing the reconstruction of early, fast evolving dynamics without neglecting the slow, more homogeneous regime that follows.

\begin{table}[ht]
\centering
\setlength{\tabcolsep}{3pt} 
\small
\begin{tabularx}{\textwidth}{lCCCC}
\toprule
\textbf{Model} & $\boldsymbol{\varepsilon_2} (\times 10^{-3})$ &
$\boldsymbol{\varepsilon_2^{\,r}} (\times 10^{-1})$ &
$\boldsymbol{W_1} (\times 10^{-6})$ & \#\textbf{params} \\
\midrule

\rowcolor{blue!18}
POD-MVAR \tiny{$(d=20,\ l=5)$} & $1.40~(0.92,2.00)$ & $0.97~(0.62,1.31)$ & $6.53~(3.89,9.94)$ & 2000 \\

\rowcolor{red!18}
DMs-MVAR \tiny{$(d=10, \ l=8)$} & $1.10~(0.37,1.80)$ & $0.69~(0.28,1.15)$ & $1.64~(0.48,3.20)$ & 800 \\

\rowcolor{blue!18}
POD-SINDy \tiny{$(d=22, \ p(1)+f)$} & $2.35~(1.54,3.36)$ & $1.59~(1.06,2.23)$ & $8.52~(6.29,11.43)$ & 1470 \\

\rowcolor{red!18}
DMs-SINDy \tiny{$(d=7, \ p_{nc}(2))$} & $1.32~(0.74,2.35)$ & $0.88~(0.50,1.56)$ & $3.59~(0.84,11.12)$ & 250 \\

\bottomrule
\end{tabularx}

\caption{Crowd dynamics paradigm. Summary of the selected POD-informed ROMs and DMs-informed ROMs performance on the test set \(X_{\text{test}}\): absolute $\ell_2$ error $\varepsilon_2$, relative $\ell_2$ error $\varepsilon_2^{\,r}$, and Wasserstein–1 distance $W_1$ in terms of average over time and 10$^{\text{th}}$–90$^{\text{th}}$ percentiles. The last column reports the number of parameters of each ROM.}
\label{tab:rom_metrics_test}
\end{table}

\begin{figure}[ht]
    \centering
    \includegraphics[width=0.9\textwidth]{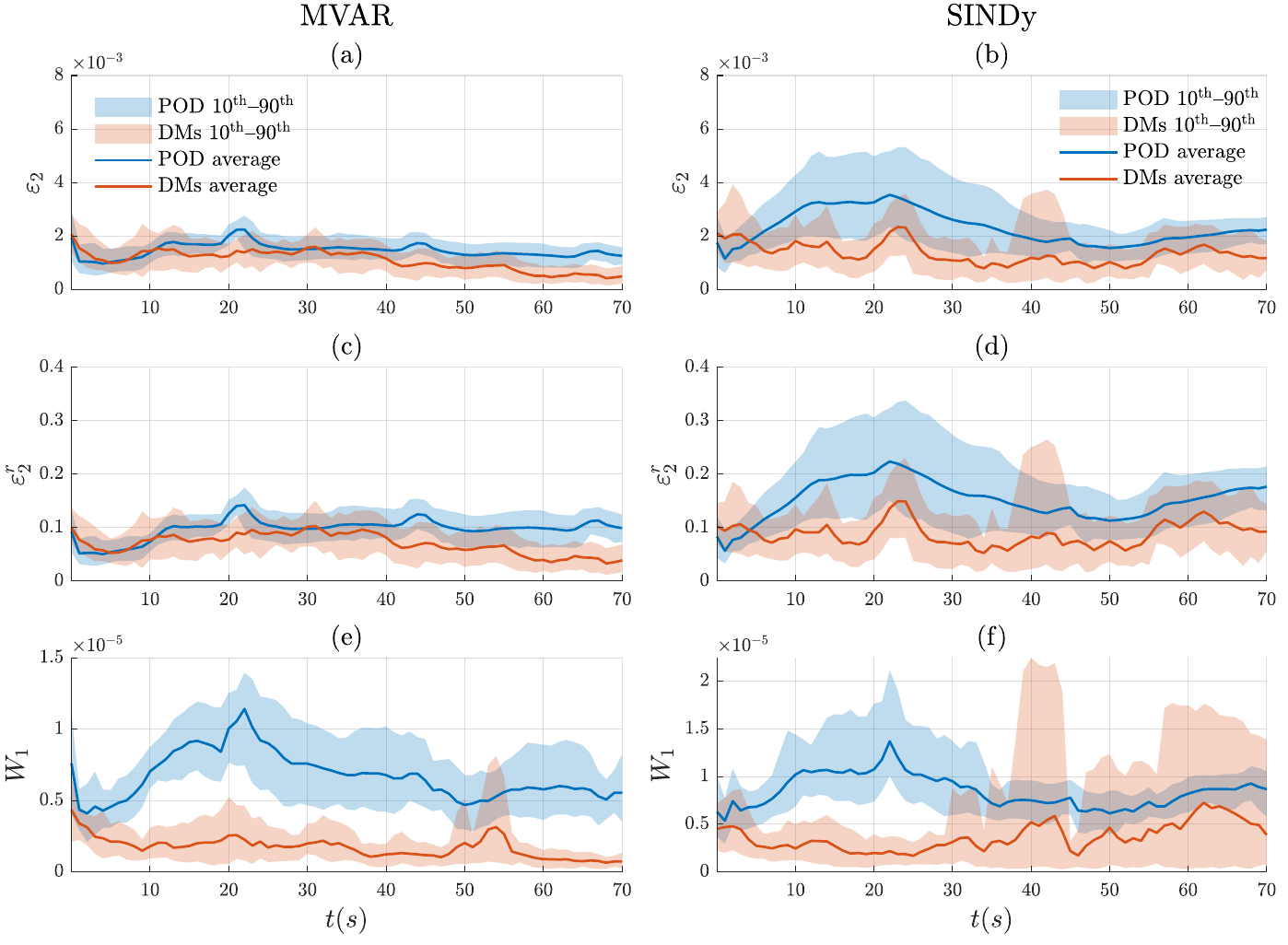}
    \caption{Crowd dynamics paradigm: time evolution of reconstruction errors on the test set \(X_{\text{test}}\). (a,b) Absolute \(\ell_2\) error (\(\varepsilon_2\)); (c,d) relative \(\ell_2\) error (\(\varepsilon_2^r\)); (e,f) Wasserstein distance (\(W_1\)) for the selected POD-MVAR $(d = 20, \ l = 5) $, DMs-MVAR $(d = 10, \ l = 8)$, POD-SINDy $(d = 22, \ p(1)+f) $, and DMs-SINDy $(d = 7, \ p_{nc}(2)) $ reduced-order-models across the \(N_{\text{test}}=40\) test simulations. In all panels, solid lines denote the mean error between the high-dimensional Hughes PDE solution (ground truth) and the ROM predictions, while shaded bands indicate the 10$^{\text{th}}$–90$^{\text{th}}$ percentiles.}
    \label{fig:MAR_var_IC}
\end{figure}

\begin{figure}
    \centering
    \includegraphics[width=0.8\textwidth]{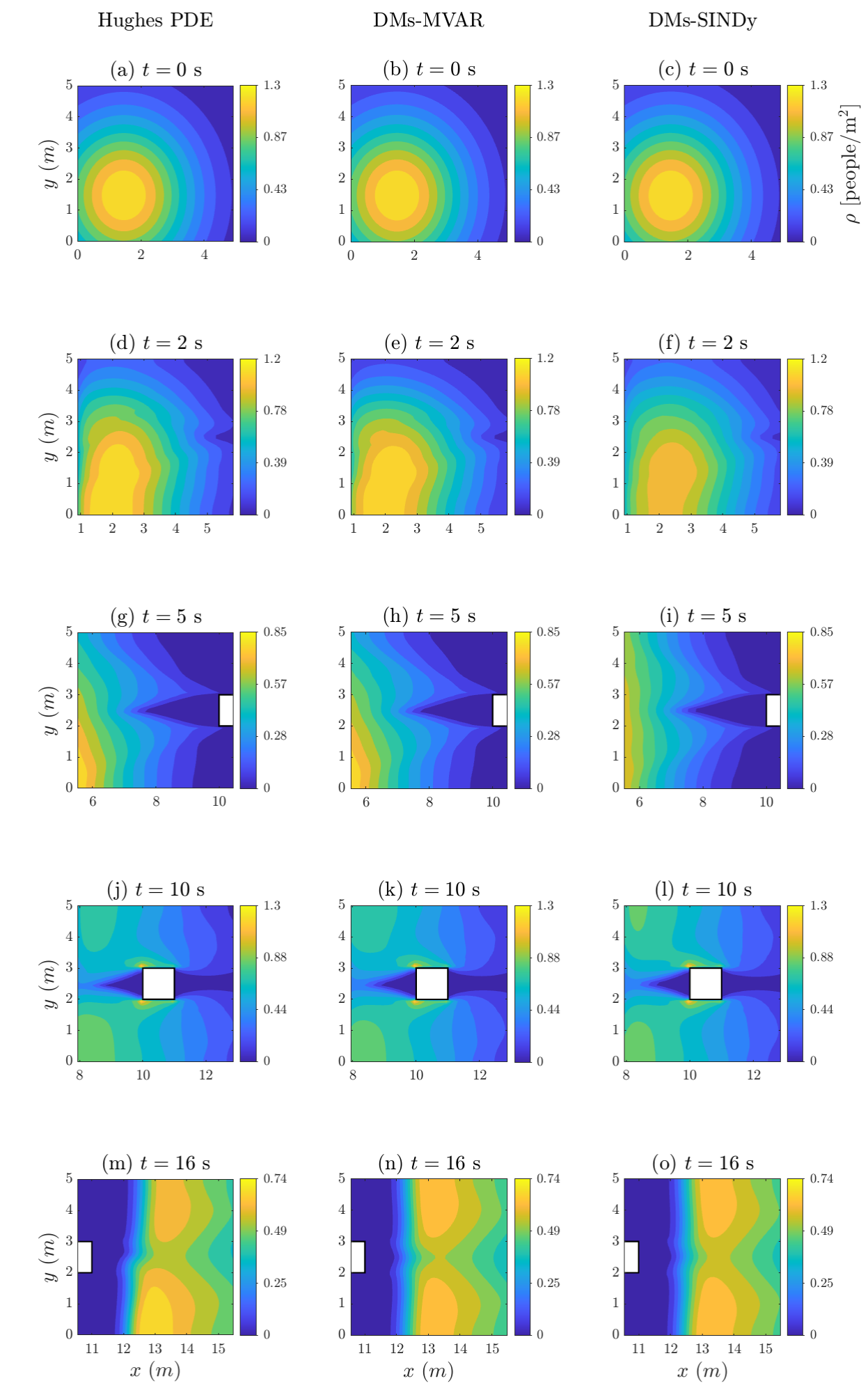}
    \caption{Comparison between the ground-truth Hughes PDE solution (left column), the DMs-informed MVAR  $(d = 10, \ l = 8)$  prediction (center column), and the DMs-informed SINDy $(d = 7,  \ p_{nc}(2)) $  prediction (third column)  at different time instants \(t\) for an unseen Gaussian initial condition \(\rho_0(x,y)\) (see Eq.~\eqref{eq:initial_condition}) with parameters \(x_0=1.5\,\text{m}\), \(y_0=1.5\,\text{m}\), \(\sigma_x=1.6\,\text{m}\), and \(\sigma_y=1.8\,\text{m}\), belonging to the test set \(X_{\text{test}}\): ((a)–(c)) \(t=0\,\text{s}\); ((d)–(f)) \(t=2\,\text{s}\); ((g)–(i)) \(t=5\,\text{s}\); ((j)–(l)) \(t=10\,\text{s}\); ((m)–(o)) \(t=16\,\text{s}\).}
    \label{fig:GT_recon_Compar_DMs}
\end{figure}

Fig. ~\ref{fig:PCA} summarizes the results of the DMs and POD embeddings computed from the \(N_\text{train}=40\) training simulations. Fig. ~\ref{fig:PCA}(a) reports the DMs eigenvalues \(\lambda_i\), while Fig. ~\ref{fig:PCA}(b) shows the mean baseline reconstruction error in terms of the Wasserstein–1 distance (\(W_1\)) for \(d=4\), \(7\), and \(10\) DMs coordinates. It can be noted that all error distributions converge to the same value for \(t>20\) s, but only the embeddings with \(d=7\) and \(d=10\) exhibit low errors during the transient phase (\(t < 20\) s). Hence, we primarily consider the first \(d=7\) DMs and later test \(d=10\) to evaluate whether additional coordinates improve latent-space predictions. Fig. ~\ref{fig:PCA}~(c) displays the cumulative distribution of the POD eigenvalues, and Fig. ~\ref{fig:PCA}(d) reports the corresponding mean reconstruction errors \(W_1\) for \(d=18\), \(20\), and \(22\) POD modes. We find that \(d=20\)–\(22\) POD modes—capturing more than 99\% of the cumulative variance—are required to achieve reconstruction accuracy comparable to that obtained with only \(d=7\)–\(10\) DMs coordinates. Overall, DMs achieve equivalent, and in the transient regime even superior, reconstruction accuracy using far fewer coordinates than POD.

In the second step, we trained two families of manifold-informed ROMs: one based on POD coordinates and another based on Diffusion Maps (DMs) coordinates. For each family, we employed both MVAR and SINDy models. A detailed analysis of the different ROMs performance on the training and validation datasets, including the dependence on the number of retained coordinates and modes, is provided in \ref{app:results_insights}. This analysis allowed us to select the four best models, based on the validation set, one for each combination of embedding algorithm (POD and DMs) and surrogate model (MVAR and SINDy). In the following, we focus on these selected models and their performance on the test set.

Fig.~\ref{fig:MAR_var_IC} shows the temporal evolution of the test-set error metrics, for the selected POD- and DMs-informed MVARs and SINDy ROMs: absolute \(\ell_2\) error (\(\varepsilon_2\)), relative \(\ell_2\) error (\(\varepsilon_2^{\,r}\)), and Wasserstein--1 distance (\(W_1\)) in the reconstructed high-dimensional space, together with 10$^{\text{th}}$–90$^{\text{th}}$ percentiles. We note that all error metrics in these diagrams are initialized at the first predicted snapshot. 

The time-averaged errors and associated percentiles for the four models are also reported in Table~\ref{tab:rom_metrics_test}. The POD-MVAR  and DMs-MVAR models are constructed with $d=20$ and $d=10$ variables and with $l=5$ and $l=8$ time delays, respectively. the POD-SINDy and DMs-SINDy models are constructed with $d=22$ and $d=7$ variables, respectively, and with a polynomial library of order one augmented with Fourier features ($p(1)+f$) and a polynomial library of order two without cross terms ($p_{nc}(2)$), respectively. As shown, the DMs-informed models  consistently outperform the POD-informed models on average across all metrics.

In Fig.~\ref{fig:GT_recon_Compar_DMs}, we provide representative snapshots of the ground-truth Hughes PDE solution (left panels), the DMs-informed MVAR (center panels), and the DMs-informed SINDy (right panels) predictions, for an unseen initial condition from the test set. As shown, the DMs-informed ROMs are able to capture spatial asymmetries and nonlinear features of the evolving field with high fidelity. It is interesting to highlight that the DMs-informed SINDy (with just one delay) and DMs-informed MVAR (with multiple delays) exhibit comparable predictive accuracy. We note that, as detailed in Tables~\ref{tab:rom_metrics_train} and ~\ref{tab:rom_metrics_validation} in \ref{app:results_insights} summarizing the results on the training and validation sets, respectively, while the POD-informed SINDy performs relatively well-even with just one delay-though still worse than the DMs-informed SINDy, the POD-informed MVAR with one delay fails entirely to capture the dynamics. 

\subsection{The Fluid dynamics Paradigm}
\label{subsec:results_fluid}

\begin{figure}
    \centering
    \includegraphics[width=0.8\textwidth]{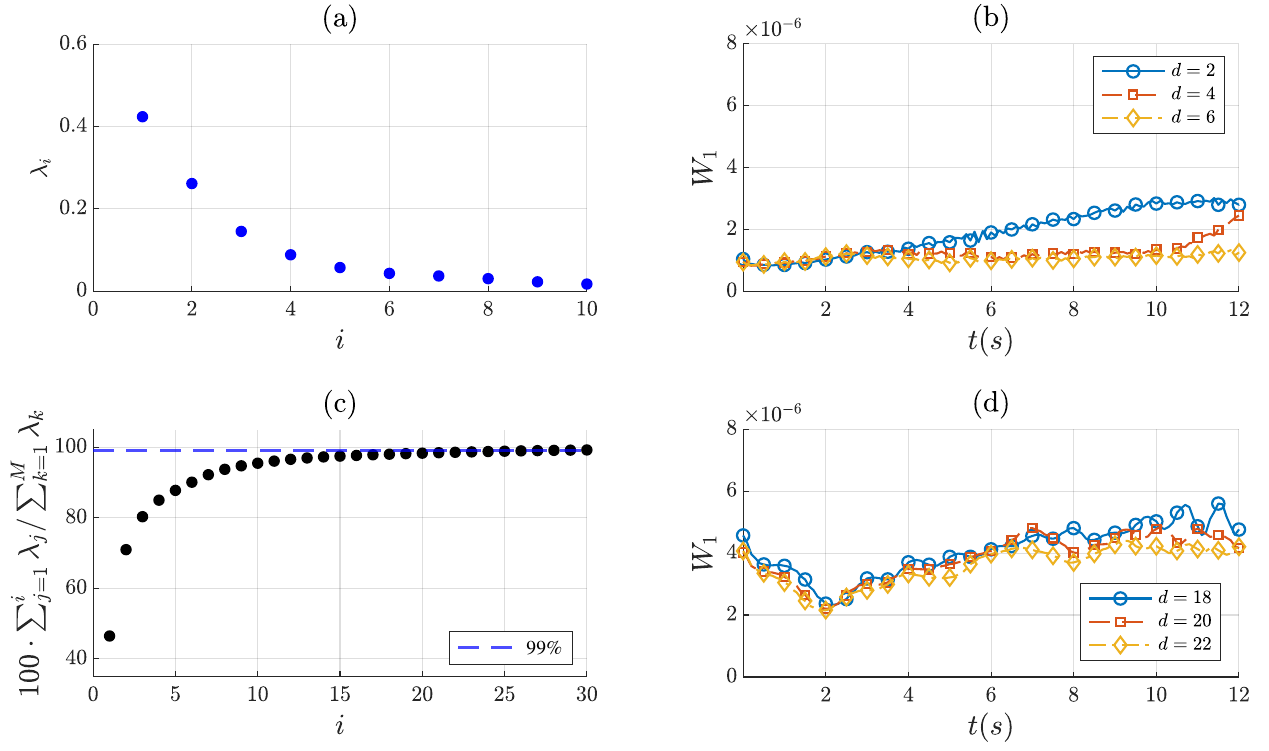}
    \caption{Fluid dynamics paradigm: (a) DMs eigenvalues \(\lambda_i\); (b) baseline mean reconstruction error \(W_1\) computed over \(d=2\), \(4\), and \(6\) DMs coordinates for the \(N_{\text{train}} = 35\) training simulations; (c) cumulative distribution of POD eigenvalues (black circles) and 99\% cumulative variance threshold (blue dashed line); (d) baseline mean reconstruction error \(W_1\) for \(d=18\), \(20\), and \(22\) POD coordinates on the training set.}
    \label{fig:PCA_fluid}
\end{figure}

Here, we evaluate the proposed methodology using numerical simulations of the Navier-Stokes flow with passive tracer. We have created \(N_{ic} = 80\) distinct initial conditions by uniformly sampling the parameters of the 2D Gaussian distribution (see Eq.~\eqref{eq:initial_condition_pinball} in \ref{app:Navier-Stokes}) within the ranges \(x_0 \in [14, 16]\) m, \(y_0 \in [-1, 1]\) m, and \(\sigma_x, \sigma_y \in [4,5]\) m, , and integrate each simulation until \(t_f = 24\)~s. From these, we constructed three datasets: a training set \(X_{\text{train}}\) containing \(N_{\text{train}} = 35\) randomly selected simulations, a validation set with \(N_{\text{val}} = 10\) samples for model selection, and a test set \(X_{\text{test}}\) including the remaining \(N_{\text{test}} = 35\) simulations for out-of-sample evaluation.
The construction of \(X_{\text{train}}\), \(X_{\text{val}}\), and \(X_{\text{test}}\)—including parameter ranges and temporal resolution—follows the protocol detailed in Section~\ref{subsec:crowd} for the crowd dynamics paradigm.

\begin{figure}
    \centering
    \includegraphics[width=0.9\textwidth]{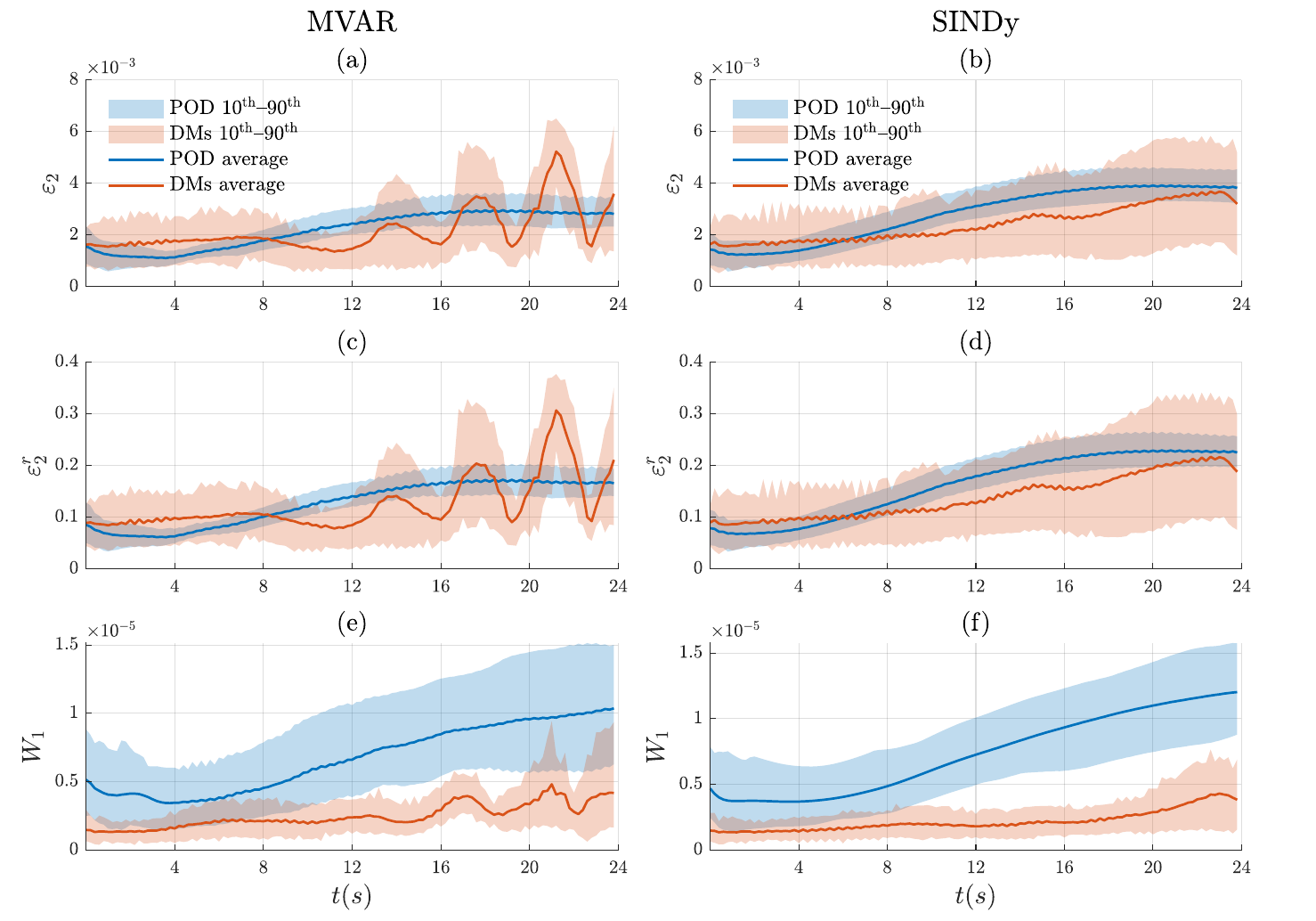}
    \caption{Fluid dynamics paradigm: time evolution of reconstruction errors on the test set \(X_{\text{test}}\). (a,b) Absolute \(\ell_2\) error (\(\varepsilon_2\)); (c,d) relative \(\ell_2\) error (\(\varepsilon_2^r\)); (e,f) Wasserstein distance (\(W_1\)) for the selected POD-MVAR $(d = 16, \ l = 3) $, DMs-MVAR $(d = 6, \ l = 3)$, POD-SINDy $(d = 20, \ p(1)+f) $, and DMs-SINDy $(d = 4, \ p(3)) $ reduced-order-models across the \(N_{\text{test}}=35\) test simulations. In all panels, solid lines denote the mean error between the high-dimensional Navier-Stokes PDE solution (ground truth) and the ROM predictions, while shaded bands indicate the 10$^{\text{th}}$–90$^{\text{th}}$ percentiles.}
    \label{fig:MAR_var_IC_pinball}
\end{figure}

In the first stage, as previously discussed, we employed both POD and DMs to learn a low-dimensional parametrization of the manifold. Fig.~\ref{fig:PCA_fluid} summarizes the results of the DMs and POD embeddings computed from the \(N_{\text{train}} = 40\) training simulations. Fig.~\ref{fig:PCA_fluid}(a) reports the DMs eigenvalues \(\lambda_i\), while Fig.~\ref{fig:PCA_fluid}(b) shows the mean baseline reconstruction error in terms of the Wasserstein--1 distance \(W_1\) for \(d = 2\), \(4\), and \(6\) DMs coordinates. Fig.~\ref{fig:PCA_fluid}(c) displays the cumulative distribution of the POD eigenvalues, and Fig.~\ref{fig:PCA_fluid}(d) reports the corresponding mean reconstruction errors \(W_1\) for \(d = 18\), \(20\), and \(22\) POD modes.

\begin{table}[ht]
\centering
\setlength{\tabcolsep}{3pt}
{
\small
\begin{tabularx}{\textwidth}{lCCCC}  
\toprule
\textbf{Model} & $\boldsymbol{\varepsilon_2} (\times 10^{-3})$ & $\boldsymbol{\varepsilon_2^{\,r}} (\times 10^{-1})$ & $\boldsymbol{W_1} (\times 10^{-6})$ & \#\textbf{params} \\
\midrule
\rowcolor{blue!18}
POD-MVAR \tiny{$(d=16,l=3)$} & $2.20~(1.74,2.74)$ & $1.26~(1.04,1.52)$ & $6.67~(3.84,10.50)$ & 768\\


 \rowcolor{red!18}
 DMs-MVAR \tiny{$(d=6, l=3)$} & $2.23~(1.19,3.42)$ & $1.27~(0.68,1.92)$ & $2.48~(1.19,4.32)$ & 108\\

\rowcolor{blue!18} 
POD-SINDy \tiny{$(d=20,p(1)+f)$} & $2.81~(2.29,3.48)$ & $1.61~(1.35,1.94)$ & $7.32~(4.69,10.36)$ & 1219\\

\rowcolor{red!18}
DMs-SINDy \tiny{$(d=4, p(3))$} & $2.41~(1.04,3.96)$ & $1.38~(0.61,2.23)$ & $2.17~(0.98,3.56)$
& 98\\

\bottomrule
\end{tabularx}
}
\caption{Fluid dynamics paradigm. Summary of the selected POD-informed ROMs and DMs-informed ROMs performance on the test set \(X_{\text{test}}\): absolute $\ell_2$ error $\varepsilon_2$, relative $\ell_2$ error $\varepsilon_2^{\,r}$, and Wasserstein–1 distance $W_1$ in terms of average over time and 10$^{\text{th}}$–90$^{\text{th}}$ percentiles. The last column reports the number of parameters of each ROM.}
\label{tab:rom_metrics_test_new}
\end{table}

\begin{figure}
    \centering
    \includegraphics[width=0.8\textwidth]{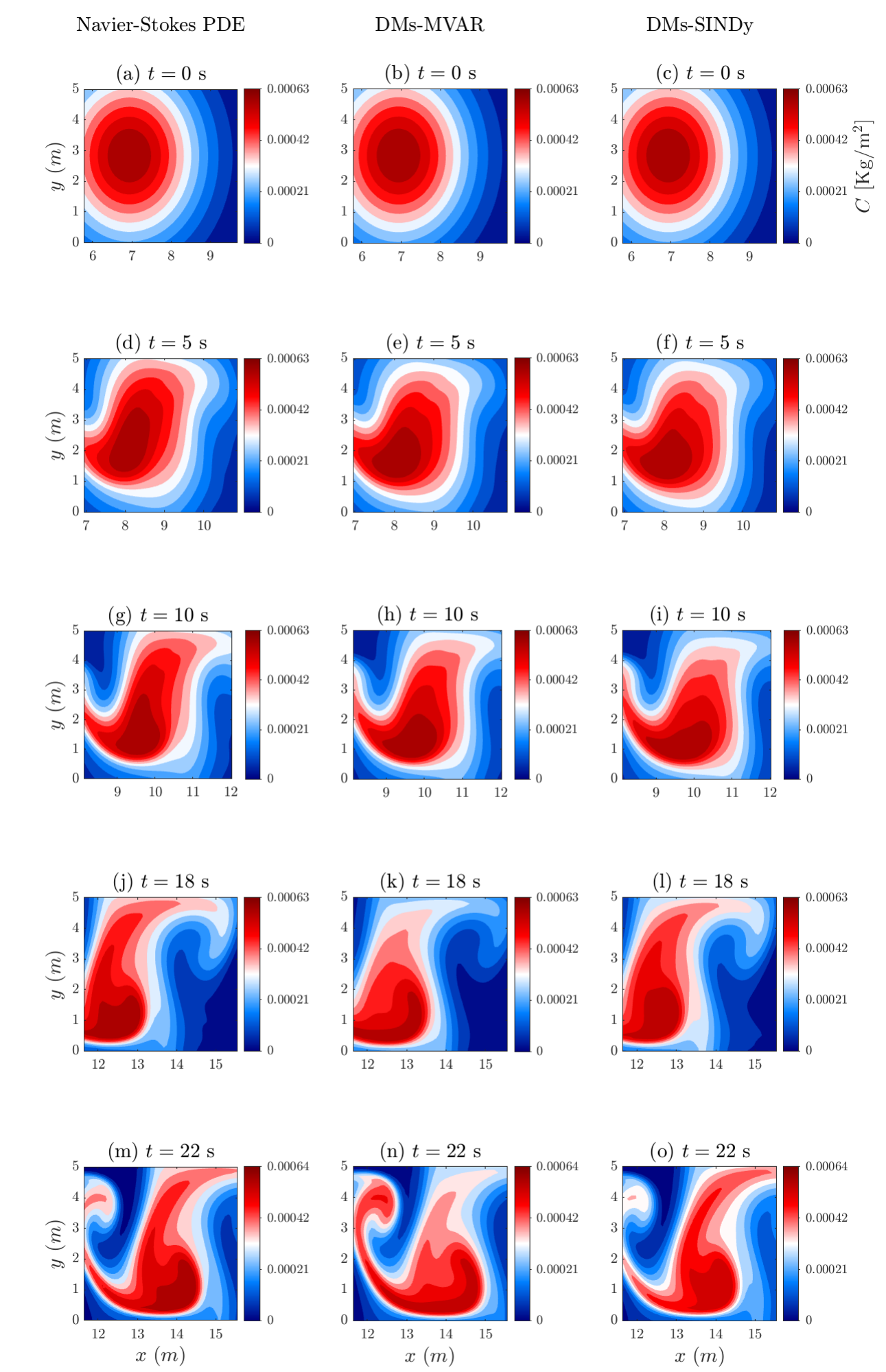}
    \caption{Comparison between the ground-truth Navier-Stokes PDE solution (left column), the DMs-informed MVAR prediction ($d=6, \ l=3$) (center column), and the DMs-informed SINDy prediction ($d=4, \ p(3)$) (right column) at different time instants \(t\) for an unseen Gaussian initial condition \(\rho_0(x,y)\) (see Eq.~\eqref{eq:initial_condition_pinball}) with parameters \(x_0=14.50\,\text{m}\), \(y_0=-0.79\,\text{m}\), \(\sigma_x=4.19\,\text{m}\), and \(\sigma_y=4.50\,\text{m}\), belonging to the test set \(X_{\text{test}}\): ((a)–(c)) \(t=0\,\text{s}\); ((d)–(f)) \(t=5\,\text{s}\); ((g)–(i)) \(t=10\,\text{s}\); ((j)–(l)) \(t=18\,\text{s}\); ((m)–(o)) \(t=22\,\text{s}\).}
    \label{fig:GT_recon_Compar_DMs_fluid}
\end{figure}

In the second step, as for the crowd dynamics paradigm, we trained again two families of manifold-informed ROMs: one based on POD embedding and another based on DMs embedding, and for each family of embeddings, we employed both MVARs and SINDy models. A detailed analysis of the different ROMs performance on the training and validation datasets, including the dependence on the number of retained coordinates and modes, is provided in \ref{app:results_insights_fluid}. From these, we  have selected the four best models according to the validation test, one for each combination of embedding algorithm (POD and DMs) and ROMs (MVAR and SINDy). In the following, we focus on these selected models and their performance on the test set.

The time-averaged errors and associated percentiles for the four models are reported in Table~\ref{tab:rom_metrics_test_new}. Note that the POD-MVAR and DMs-MVAR ROMs are constructed using $d=16$ and $d=6$ variables and  $l=3$ time delays, respectively. The POD-SINDy, and DMs-SINDy ROMs are constructed with $d=20$, and only $d=4$ variables, respectively, and with a polynomial library of order one augmented with Fourier features ($p(1)+f$) and a polynomial library of order three ($p(3)$).

Fig.~\ref{fig:MAR_var_IC_pinball} shows the temporal evolution of the test-set error metrics for the selected POD- and DMs-informed MVAR and SINDy ROMs: absolute \(\ell_2\) error (\(\varepsilon_2\)), relative \(\ell_2\) error (\(\varepsilon_2^{\,r}\)), and Wasserstein--1 distance (\(W_1\)) in the reconstructed high-dimensional space, together with 10$^{\text{th}}$–90$^{\text{th}}$ percentiles. Note that all curves are initialized at the first predicted snapshot.
As shown, the DMs-informed MVAR models exhibit oscillatory behavior despite comparable short-term accuracy in all metrics. On average, the DMs-informed SINDy models are more accurate with respect to $W_1$, and comparable to the other two metrics albeit with wider 10$^{\text{th}}$–90$^{\text{th}}$ percentiles spreads. We also note that the POD-informed  ROMs  (both MVARs and SINDy) result in a divergent $W_1$ metric.

Finally, in Fig.~\ref{fig:GT_recon_Compar_DMs_fluid}, we provide representative snapshots of the ground-truth Navier-Stokes PDE solution (left panels), the DMs-informed MVAR (center panels), and the DMs-informed SINDy (right panels) predictions, for an unseen initial condition from the test set. As illustrated by these snapshots and throughout the test set, the DMs-informed SINDy captures more accurately the spatio-temporal dynamics, a feature that it is not fully reflected by the global metrics reported in Table \ref{tab:rom_metrics_test_new}. In fact, the DMs-informed MVAR prediction shows a spatio-temporal shift relative to the ground-truth solution, which explains the oscillatory behavior observed in all the metrics reported in Fig.~\ref{fig:MAR_var_IC_pinball} (red curves in panels (a), (c), and (e)). This can be appreciated in particular by comparing Fig.~\ref{fig:GT_recon_Compar_DMs_fluid}(m) and Fig.~\ref{fig:GT_recon_Compar_DMs_fluid}(n), displaying the ground-truth Navier–Stokes solution and the DMs-MVAR prediction, respectively, at 
$t=22$ s, which is the time instant of maximum reconstruction error.


\section{Conclusions}
\label{sec:conclusions}
For systems with conservation laws, learning the full high-dimensional solution operator for a mass-conserving PDE is not trivial \cite{jagtap2020conservative,huang2023limitations,chen2024learning,patsatzis2025gorinns}: small model errors compound under long-time integration, producing significant drifts or even blow-ups.
The proposed PDE-free framework overcomes this limitation and applies broadly to systems that can, in principle, be described by PDEs—including those with hidden or unobserved variables—by avoiding the need to learn the full spatio-temporal operator. Instead, we construct reduced-order models (ROMs) in latent spaces, recovering an effective ODE-like operator using nonlinear manifold learning techniques and in particular Diffusion maps (DMs). This approach is far more computationally efficient than learning the full PDE solution operator, a task that suffers from the curse of dimensionality. Once trained, the manifold-informed ROM allows reconstruction of the original spatio-temporal fields via the manifold \textit{mass-preserving} pre-image problem, reproducing the high-dimensional solution operator \emph{implicitly} and \emph{on demand}. 
Hence, with such ROMs, one can perform accurate long-horizon simulations.

For these ROMs to achieve high accuracy over long time-horizons, the latent dynamics must be represented by a well-parameterized embedding of the underlying manifold.  Here and in other works~\cite{della2024learning,della2025surrogate}, we have shown that for that task, DMs outperform POD—which remains dominant in the vast majority of such studies. Consequently, significantly fewer coordinates are required to parametrize the low-dimensional latent space where robust and accurate ROMs can be constructed with reduced complexity. However, this advantage over POD comes with a challenge: the pre-image problem—mapping latent coordinates back to high-dimensional fields—is ill-posed and usually computationally more expensive. Crucially, mass conservation must be enforced consistently when lifting back to physical space by using an appropriate lifting scheme, such as the $k$-NN algorithm with convex interpolation~\cite{chin2024enabling}, which, as we have proven here, preserves mass positivity; as we also prove here, mass preservation is also achieved with POD lifting, even though positivity is not guaranteed.
Future work will address systems with multiple basins of attraction— and comparing our manifold-DMs-informed ROMs against autoencoders, neural operators and other ML architectures to quantify robustness, generalization and conservation properties.
Another direction is to handle cases where high-dimensional statistics are functions of low-dimensional variables (i.e., low-to-high probability distributions) \cite{giovanis2025generative}, thus identifying the underlying
SDEs, that preserve distributional constraints.

\appendix
\label{sec:appendix}

\section{Numerical simulation of the Hughes model}
\label{app:numerical_scheme}

High-dimensional crowd density data were generated by numerically solving the Hughes model (Eqs.~\eqref{hughes_a}–\eqref{hughes_b}) with an in-house finite-volume scheme implemented in Python. The spatial domain is discretized into \(N = N_x \times N_y = 200 \times 50\) cells, with $\rho_{i,j}^n$ denoting the cell-averaged density at $(x_i, y_j)$ and time $t^n$. To capture different crowd regimes, $N_{ic} = 100$ simulations with varying initial conditions are performed, each producing $N_t = 700$ snapshots spaced by $\Delta t = 0.1$~s. Snapshots are collected as column vectors $\mathbf{x}_m \in \mathbb{R}^N$ and assembled in the snapshot matrix $\mathbf{X} = \{\mathbf{x}_m\}_{m=1,\dots,M} \in \mathbb{R}^{N \times M}$, with $M = N_t \cdot N_{ic}$.

The domain is a rectangular corridor $\Omega$ with a central square obstacle $\Omega_{\text{obs}}$, length $L = 20$~m, height $H = 5$~m, and obstacle side $L_{\text{obs}} = 1$~m (see Fig.~\ref{fig:domain}(a) in Section~\ref{sec:layout}). Initial conditions are 2D Gaussian density distributions:
\begin{equation}
\rho_0(x,y) = \Gamma_0 \exp\left[-\frac{(x-x_0)^2}{2\sigma_x^2} - \frac{(y-y_0)^2}{2\sigma_y^2}\right],
\label{eq:initial_condition}
\end{equation}
with the total mass,
\begin{equation}
\label{eq_mass}
m = \iint_\Omega \rho_0(x,y) \, dx\,dy,
\end{equation}
kept constant across all initial conditions by adjusting $\Gamma_0$.

The Hughes PDE model (Eqs.~\eqref{hughes_a}--\eqref{hughes_b}) is complemented with the following boundary conditions prescribed on the inlet ($\Gamma_{\text{in}}$), outlet ($\Gamma_{\text{out}}$), the domain walls ($\Gamma_w$), and the obstacle boundary ($\Gamma_{\text{obs}}$):

\begin{equation}
\left\{
\begin{aligned}
\left( \rho f(\rho) \frac{\partial \phi/\partial n}{\|\nabla \phi\|} \right) \cdot \mathbf{n}(x,y) &= 0 && \text{on } (\Gamma_{\text{obs}} \cup \Gamma_w) \times [0,t_f],\\
\rho(0, y, t) &= \rho(L, y, t) && y \in [0,H],\ t \in [0,t_f],\\
\phi &= 0 && \text{on } \Gamma_{\text{out}} \times [0,t_f],
\end{aligned}
\right.
\label{bcic}
\end{equation}
where $\partial \Omega = \Gamma_{\text{in}} \cup \Gamma_{\text{out}} \cup \Gamma_w \cup \Gamma_{\text{obs}}$, and $\mathbf{n}$ is the local outward unit normal. Homogeneous Neumann conditions enforce zero flux at the walls, while periodicity in $x$ ensures inflow–outflow continuity. The potential is set to zero at the exit. Note that each simulation is run until approximately four loops of the domain are completed. 

The numerical solution of the Hughes model (Eqs.~\eqref{hughes_a}--\eqref{hughes_b}) is carried out using a finite-volume scheme combined with the Fast Sweeping Method~\cite{Zhao2004FastSweeping} to efficiently solve the nonlinear Eikonal equation~\eqref{hughes_b}. The approach outlined here specifies the spatial and temporal discretization of the governing equations, the evaluation of numerical fluxes via Godunov's method, and the stability condition adopted for time-step selection.

At each time step, given the current density field \( \rho_{i,j}^n \), the right-hand side of the Eikonal equation is evaluated, and the potential \( \phi_{i,j}^n \) is computed using the Fast Sweeping Method~\cite{Zhao2004FastSweeping}. Once the potential is available, the system~\eqref{hughes_a}--\eqref{hughes_b} reduces to a scalar nonlinear conservation law:
\begin{equation}
    \frac{\partial \rho}{\partial t} + \frac{\partial F_x(\rho,\phi)}{\partial x} + \frac{\partial F_y(\rho,\phi)}{\partial y} = 0,
    \label{veldens_new}
\end{equation}
where \( F_x(\rho, \phi) \) and \( F_y(\rho, \phi) \) denote the density fluxes along the \( x \)- and \( y \)-directions:
\begin{align}
F_x(\rho, \phi) &= \rho U_x(\rho, \phi) = \rho g(\rho) v_f \left(1 - \frac{\rho}{\rho_m} \right) \frac{\phi_x}{\sqrt{\phi_x^2 + \phi_y^2}}, \label{eq:fluxes_x}\\
F_y(\rho, \phi) &= \rho U_y(\rho, \phi) = \rho g(\rho) v_f \left(1 - \frac{\rho}{\rho_m} \right) \frac{\phi_y}{\sqrt{\phi_x^2 + \phi_y^2}}, \label{eq:fluxes_y}
\end{align}
with $\phi_x = \partial \phi/\partial x$ and $\phi_y = \partial \phi/\partial y$ evaluated by first-order forward finite differences of the discrete potential field.

The density field is then updated in conservative form:
\begin{equation}
\rho_{i,j}^{n+1} = \rho_{i,j}^{n}
- \frac{\Delta t}{\Delta x} \left( \hat{F}_{i+\frac{1}{2},j}^n - \hat{F}_{i-\frac{1}{2},j}^n \right)
- \frac{\Delta t}{\Delta y} \left( \hat{F}_{i,j+\frac{1}{2}}^n - \hat{F}_{i,j-\frac{1}{2}}^n \right),
\label{eq:fv_update}
\end{equation}
where \( \Delta x = \Delta y \) denotes the uniform grid spacing.

The intercell fluxes are computed using Godunov's method. In the \(x\)-direction:
\begin{align}
\hat{F}^n_{i+\frac{1}{2},j} &=
\begin{cases}
\min\limits_{\rho^n_{i,j} \leq \theta \leq \rho^n_{i+1,j}} F_x(\theta) & \text{if } \rho^n_{i,j} \leq \rho^n_{i+1,j}, \\
\max\limits_{\rho^n_{i+1,j} \leq \theta \leq \rho^n_{i,j}} F_x(\theta) & \text{otherwise},
\end{cases} \\
\hat{F}^n_{i-\frac{1}{2},j} &=
\begin{cases}
\min\limits_{\rho^n_{i-1,j} \leq \theta \leq \rho^n_{i,j}} F_x(\theta) & \text{if } \rho^n_{i-1,j} \leq \rho^n_{i,j}, \\
\max\limits_{\rho^n_{i,j} \leq \theta \leq \rho^n_{i-1,j}} F_x(\theta) & \text{otherwise},
\end{cases}
\end{align}
and analogously in the \(y\)-direction:
\begin{align}
\hat{F}^n_{i,j+\frac{1}{2}} &=
\begin{cases}
\min\limits_{\rho^n_{i,j} \leq \theta \leq \rho^n_{i,j+1}} F_y(\theta) & \text{if } \rho^n_{i,j} \leq \rho^n_{i,j+1}, \\
\max\limits_{\rho^n_{i,j+1} \leq \theta \leq \rho^n_{i,j}} F_y(\theta) & \text{otherwise},
\end{cases} \\
\hat{F}^n_{i,j-\frac{1}{2}} &=
\begin{cases}
\min\limits_{\rho^n_{i,j-1} \leq \theta \leq \rho^n_{i,j}} F_y(\theta) & \text{if } \rho^n_{i,j-1} \leq \rho^n_{i,j}, \\
\max\limits_{\rho^n_{i,j} \leq \theta \leq \rho^n_{i,j-1}} F_y(\theta) & \text{otherwise}.
\end{cases}
\end{align}

The fluxes \( F_x(\theta) \) and \( F_y(\theta) \) in Eqs.~\eqref{eq:fv_update} are evaluated at frozen coefficients \( \phi_x, \phi_y \). The optimization is carried out with respect to \( \theta \in [\rho_L,\rho_R] \), where \( \rho_L \) and \( \rho_R \) are the left and right states at the interface.

Finally, the time step $\Delta t$ is chosen according to a Courant--Friedrichs--Lewy (CFL) condition:
\begin{equation}
\Delta t = CFL \cdot \min \left( \frac{\Delta x}{\max |U_x|}, \frac{\Delta y}{\max |U_y|} \right),
\end{equation}
where the maxima of $|U_x|$ and $|U_y|$ (see Eqs.~\eqref{eq:fluxes_x}--\eqref{eq:fluxes_y}) are taken over all grid points. In all simulations, a Courant number of $CFL = 0.25$ was adopted.

\section{Crowd dynamics paradigm: optimal model selection}
\label{app:results_insights}

\begin{table}[ht]
\centering
\setlength{\tabcolsep}{2pt}
{\rowcolors{1}{white}{cyan!8} 
\small
\begin{tabularx}{\textwidth}{lCCC}
\toprule
\textbf{Model} 
& $\boldsymbol{\varepsilon_2}$ $(\times 10^{-3})$
& $\boldsymbol{\varepsilon_2^{\,r}}$ $(\times 10^{-1})$
& $\boldsymbol{W_1}$ $(\times 10^{-6})$ \\
\midrule
POD-MVAR $(d=18,\ l=5)$ 
& $1.60~(1.00,2.20)$ 
& $1.07~(0.70,1.42)$  
& $7.15~(3.89,11.60)$ \\

POD-MVAR $(d=20,\ l=5)$ 
& $1.40~(1.00,1.90)$ 
& $0.97~(0.66,1.12)$  
& $6.22~(3.74,9.49)$ \\

POD-MVAR $(d=22,\ l=5)$ 
& $1.40~(0.91,1.80)$ 
& $0.94~(0.59,1.26)$  
& $6.47~(3.82,9.52)$ \\

POD-MVAR $(d=22, l=1)$ & $27.3~(24.4,30.2)$ 
& $18.7~(17.0,20.6)$ 
& $122.0~(106.8,140.9)$\\

POD-SINDy $(d=18,\ \mathrm{p}(1)+f)$ 
& $2.53~(1.78,3.59)$ 
& $1.69~(1.22,2.34)$ 
& $10.06~(7.78,13.38)$ \\

POD-SINDy $(d=20,\ \mathrm{p}(1)+f)$ 
& $2.58~(1.81,3.70)$ 
& $1.71~(1.24,2.41)$ 
& $10.29~(7.75,13.86)$ \\

POD-SINDy $(d=22,\ \mathrm{p}(1)+f)$ 
& $2.46~(1.60,3.57)$ 
& $1.64~(1.10,2.34)$ 
& $8.75~(6.31,12.25)$ \\


DMs-MVAR $(d=7,\ l=8)$ 
& $1.30~(0.42,2.30)$ 
& $0.87~(0.31,1.50)$  
& $2.24~(0.56,4.32)$ \\

DMs-MVAR $(d=10,\ l=8)$ 
& $1.00~(0.37,1.70)$ 
& $0.67~(0.26,1.22)$  
& $1.67~(0.50,3.18)$ \\


DMs-SINDy $(d=7,\ \mathrm{pnc}(2))$ 
& $1.44~(0.63,2.48)$ 
& $0.97~(0.43,1.67)$ 
& $4.63~(0.70,11.84)$ \\

DMs-SINDy $(d=10,\ \mathrm{pnc}(2))$ 
& $1.96~(0.78,3.33)$ 
& $1.32~(0.53,2.22)$ 
& $3.86~(1.02,8.08)$ \\
\bottomrule
\end{tabularx}
} 
\caption{Crowd dynamics paradigm: summary of the manifold-informed ROMs performance on the training set \(X_{\text{train}}\). Columns report absolute $\ell_2$ error ($\varepsilon_2$), relative $\ell_2$ error ($\varepsilon_2^{\,r}$), and Wasserstein--1 distance ($W_1$) in terms of average over time and 10$^{\text{th}}$–90$^{\text{th}}$ percentiles.}
\label{tab:rom_metrics_train}
\end{table}

\begin{table}[ht]
\centering
\setlength{\tabcolsep}{2pt}
{
\rowcolors{2}{cyan!8}{white} 
\small
\begin{tabularx}{\textwidth}{lCCC}  
\toprule
\textbf{Model} & $\boldsymbol{\varepsilon_2} (\times 10^{-3})$ & $\boldsymbol{\varepsilon_2^{\,r}} (\times 10^{-1})$ & $\boldsymbol{W_1}  (\times 10^{-6})$ \\
\midrule
POD-MVAR $(d=18,\ l=5)$ & $1.50~(0.95,2.10)$ & $1.01~(0.67,1.32)$  & $5.67~(3.03,9.21)$ \\
\rowcolor{green!18} 
POD-MVAR $(d=20,\ l=5)$ & $1.30~(0.91,1.80)$ & $0.91~(0.57,1.18)$  & $4.75~(2.70,7.30)$ \\
POD-MVAR $(d=22,\ l=5)$ & $1.30~(0.79,1.80)$ & $0.88~(0.53,1.17)$  & $5.03~(2.82,7.96)$ \\

POD-MVAR $(d=22, \ l=1)$ & $27.50~(26.0,29.0)$ & $19.0~(18.0,20.0)$ & $127.1~(118.0,136.0)$\\

POD-SINDy $(d=18, \ p(1)+f)$ & $2.23~(1.85,2.64)$ & $1.52~(1.28,1.79)$ & $9.36~(8.21,10.72)$\\
POD-SINDy $(d=20, \ p(1)+f)$ & $2.21~(1.84,2.66)$ & $1.50~(1.25,1.79)$ & $9.51~(8.26,10.99)$\\
\rowcolor{green!18} 
POD-SINDy $(d=22, \ p(1)+f)$ & $2.03~(1.60,2.49)$ & $1.40~(1.11,1.70)$ & $7.67~(6.48,9.08)$\\
DMs-MVAR $(d=7,\ l=8)$  & $1.20~(0.34, 2.30)$ & $0.81~(0.24,1.49)$  & $2.18~(0.41,4.79)$\\
\rowcolor{green!18} 
DMs-MVAR $(d=10,\ l=8)$ & $ 0.91~(0.32, 1.60)$ & $0.60~(0.23,1.02)$  & $1.39~(0.36,3.11)$\\
\rowcolor{green!18}
DMs-SINDy $(d=7, \ pnc(2))$ & $1.14~(0.68,1.70)$ & $0.77~(0.46,1.14)$ & $2.66~(0.72,6.51)$\\
DMs-SINDy $(d=10, \ pnc(2))$ & $1.31~(0.54,2.24)$ & $0.88~(0.36,1.51)$ & $2.69~(0.90,5.29)$\\
DMs-MVAR $(d=7, \ l=5)$ & $1.21~(0.72,1.73)$ & $0.82~(0.50,1.16)$ & $1.93~(0.79,3.22)$\\
\bottomrule
\end{tabularx}
}
\caption{Crowd dynamics paradigm: summary of manifold-informed ROMs performance on the validation set \(X_{\text{val}}\). Columns report absolute $\ell_2$ error ($\varepsilon_2$), relative $\ell_2$ error ($\varepsilon_2^{\,r}$), and Wasserstein–1 distance ($W_1$) in terms of average over time and 10$^{\text{th}}$–90$^{\text{th}}$ percentiles. Rows shaded in green highlight the selected POD-informed and DMs-informed ROMs.}
\label{tab:rom_metrics_validation}
\end{table}

We detail here the procedure used to select the optimal reduced-order models for the crowd dynamics paradigm. In particular, we investigate the dependence of reconstruction and forecasting accuracy on the number of retained POD modes and DMs coordinates, and compare the performance of the MVAR and SINDy learning algorithms applied to these latent coordinates.

For the MVAR models, the lag order \(l\) was selected using the Bayesian Information Criterion (BIC). The BIC yields \(l=5\) for POD-informed ROMs with \(d=18\), \(20\), and \(22\) modes, and \(l=8\) for DMs-informed ROMs with \(d=7\) and \(10\) coordinates. Tables~\ref{tab:rom_metrics_train} and~\ref{tab:rom_metrics_validation} summarize performance on the training and validation datasets, respectively, in terms of absolute error (\(\varepsilon_2\)), relative error (\(\varepsilon_2^{\,r}\)), and Wasserstein–1 distance ($W_1$). All metrics are averaged over time with 10$^{\text{th}}$–90$^{\text{th}}$ percentiles. For the reconstruction of the DMs-informed ROMs predictions in the high-dimensional space, we have used the $k$-NN algorithm with $k=8$.

Within the DMs-informed ROMs, reconstruction accuracy of the MVARs improves with the number of retained coordinates. In contrast, for the SINDy, more parsimonious models yield better accuracy, with the best performance achieved for \(d=7\). For the MVARs, although \(d=7\) and \(d=10\) exhibit comparable projection-stage accuracy, predictive errors are consistently lower for \(d=10\) across all metrics and datasets.

For the POD-informed ROMs, both MVAR and SINDy models show increasing accuracy as the number of retained modes grows. The best performance is achieved with \(d=20\) modes for the MVAR and \(d=22\) modes for the SINDy-ROM, based on the $W_1$ metric evaluated on both the training and validation datasets. Accordingly, we selected \(d=10\) as the optimal DMs-informed MVAR ROM and \(d=20\) as the optimal POD-informed MVAR ROM for test-set evaluation, and \(d=4\) as the optimal DMs-informed SINDy ROM and \(d=22\) as the optimal POD-informed SINDy ROM. Results obtained with these four selected models have been presented and discussed in Section~\ref{subsec:crowd}. 


\section{Numerical simulation of the fluidic pinball}
\label{app:Navier-Stokes}

The Navier-Stokes equations~\eqref{eq:continuity}--\eqref{eq:momentum_v} are solved by using a projection method: a provisional velocity field is first computed neglecting the pressure gradient, and subsequently projected onto the space of divergence-free fields by solving a pressure Poisson equation. Once the velocity field is obtained, it is used to advect the tracer concentration by solving Eq.~\eqref{eq:advection_f} with the second-order Bell--Colella--Glaz (BCG) scheme. Details on the numerical implementation in the open-source code \texttt{BASILISK} can be found in Popinet~\cite{Popinet2003}. All computations are performed on a uniform Cartesian grid with spacing $\Delta x = \Delta y = 0.2 D$, corresponding to $N = N_x \times N_y = 15124$ grid cells. As shown in \cite{DellaPia_diffusion_2024}, this resolution is sufficient to ensure grid-independent results in the periodic flow regime of the fluidic pinball.

The computational domain is a rectangle of dimensions $L_x = 52D$ and $L_y = 12D$ ($D = 1$ m), from which the interior of the cylinders is excluded (see Fig.~\ref{fig:domain}(b) in Section~\ref{sec:pinball}). A Cartesian coordinate system $\mathcal{O}xy$ is adopted, with the origin located at the midpoint between the two rear cylinders. Since the cylinders are fixed in the present study, no-slip boundary conditions are enforced on their surfaces. At the inlet (left boundary), a uniform velocity profile is prescribed ($u = U_\infty$, $v = 0$), while a standard free-outflow condition is imposed at the outlet (right boundary). Homogeneous Neumann boundary conditions are applied on the remaining boundaries.

The initial velocity field is prescribed as
\begin{equation}
\label{eq:init}
u(t=0) = U_\infty + u'_0, \qquad v(t=0) = v'_0,
\end{equation}
where $u'_0$ and $v'_0$ are random perturbations modeled as white noise with amplitude equal to $10\%$ of the free-stream velocity ($U_\infty = 1$ m/s). The tracer concentration is initially set to zero throughout the domain.

After the spatio-temporal periodic oscillations are fully established, the tracer is initialized as a two-dimensional Gaussian distribution,
\begin{equation}
C_0(x,y) = \bar{C}_0 
\exp\!\left[-\frac{(x-x_0)^2}{2\sigma_x^2} - \frac{(y-y_0)^2}{2\sigma_y^2}\right],
\label{eq:initial_condition_pinball}
\end{equation}
as in the previous case study. The total initial tracer mass, defined as the spatial integral of $C_0(x,y)$ over the domain (see Eq.~\eqref{eq_mass} in \ref{app:numerical_scheme}), is kept constant across all simulations by appropriately adjusting the amplitude $\bar{C}_0$.

To construct the dataset, $N_{ic} = 80$ simulations are performed with different initial tracer distributions, obtained by varying the parameters within the ranges \(x_0 \in [14, 16]\) m, \(y_0 \in [-1, 1]\) m, and \(\sigma_x, \sigma_y \in [4,5]\) m. Since the passive tracer is the mass-preserving quantity of interest in this case study, each simulation is advanced up to $t_f = 24$ s, beyond which the tracer leaves the domain and mass conservation is no longer satisfied.
For each initial condition, $N_t = 120$ snapshots of the tracer field are stored at uniform time intervals $\Delta t = 0.2$ s. The snapshots are reshaped into column vectors $\mathbf{x}_m \in \mathbb{R}^N$ and assembled into the snapshot matrix $\mathbf{X} = \{\mathbf{x}_m\}_{m=1,\dots,M} \in \mathbb{R}^{N \times M}$, with $M = N_t \cdot N_{ic}$.

\section{Fluid dynamics paradigm: optimal model selection}
\label{app:results_insights_fluid}

\begin{table}[ht]
\centering
\setlength{\tabcolsep}{2pt}
{
\rowcolors{2}{cyan!8}{white} 
\small
\begin{tabularx}{\textwidth}{lCCC}  
\toprule
\textbf{Model} & $\boldsymbol{\varepsilon_2} (\times 10^{-3})$ & $\boldsymbol{\varepsilon_2^{\,r}} (\times 10^{-1})$ & $\boldsymbol{W_1}  (\times 10^{-6})$ \\
\midrule
POD-MVAR $(d=16,l=3)$ & $2.16~(1.66,2.80)$ & $1.24~(0.99,1.56)$ & $6.16~(3.60,9.41)$\\

POD-MVAR $(d=18,l=2)$ & $3.13~(2.64,3.84)$ & $1.80~(1.57,2.18)$ & $12.36~(10.07,15.11)$\\

POD-MVAR $(d=20,l=2)$ & $3.82~(3.41,4.39)$ & $2.20~(2.01,2.48)$ & $17.11~(14.78,19.70)$\\

POD-SINDy $(d=16,p(1)+f)$ & $2.96~(2.41,3.71)$ & $1.71~(1.44,2.10)$ & $7.52~(4.65,11.10)$\\

POD-SINDy $(d=18,p(1)+f)$ & $2.76~(2.21,3.57)$ & $1.59~(1.31,2.01)$ & $7.30~(4.68,10.37)$\\

POD-SINDy $(d=20,p(1)+f)$ & $2.69~(2.23,3.39)$ & $1.55~(1.32,1.91)$ & $6.85~(4.53,9.55)$\\


DMs-MVAR $(d=4, l=5)$ & $2.33~(1.05,3.78)$ & $1.34~(0.62,2.16)$ & $2.50~(0.78,4.55)$\\

DMs-MVAR $(d=6, l=3)$ & $1.94~(0.89,3.11)$ & $1.12~(0.52,1.77)$ & $2.14~(0.90,3.61)$\\


DMs-SINDy $(d=4, p(3))$ & $1.97~(0.68,3.64)$ & $1.14~(0.40,2.11)$ & $1.82~(0.67,3.01)$\\

DMs-SINDy $(d=6, p(1) + f)$ & $2.57~(1.64,3.78)$ & $1.48~(0.95,2.15)$ & $4.63~(1.74,7.24)$\\

\bottomrule
\end{tabularx}
}
\caption{Fluid dynamics paradigm: summary of manifold-informed ROMs performance on the training set \(X_{\text{train}}\). Columns report absolute $\ell_2$ error ($\varepsilon_2$), relative $\ell_2$ error ($\varepsilon_2^{\,r}$), and Wasserstein–1 distance ($W_1$) in terms of average over time and 10$^{\text{th}}$–90$^{\text{th}}$ percentiles. Rows shaded in green highlight the selected POD-informed and DMs-informed ROMs.}
\label{tab:rom_metrics_train_new}
\end{table}

\begin{table}[ht]
\centering
\setlength{\tabcolsep}{2pt}
{
\rowcolors{2}{cyan!8}{white} 
\small
\begin{tabularx}{\textwidth}{lCCC}  
\toprule
\textbf{Model} & $\boldsymbol{\varepsilon_2} (\times 10^{-3})$ & $\boldsymbol{\varepsilon_2^{\,r}} (\times 10^{-1})$ & $\boldsymbol{W_1}  (\times 10^{-6})$ \\
\midrule
\rowcolor{green!18} 
POD-MVAR $(d=16, l=3)$ & $3.56~(2.66,4.49)$ & $2.05~(1.54,2.55)$ & $8.55~(4.26,15.01)$\\

POD-MVAR $(d=18, l=2)$ & $4.28~(3.25,5.16)$ & $2.48~(1.87,3.05)$ & $14.65~(11.08,20.26)$\\

POD-MVAR $(d=20, l=2)$ & $4.89~(3.89,5.79)$ & $2.84~(2.25,3.41)$ & $19.14~(15.48,24.48)$\\

POD-SINDy $(d=16,p(1) + f)$ & $4.03~(3.19,4.99)$ & $2.33~(1.84,2.82)$ & $9.21~(4.58,17.36)$\\

POD-SINDy $(d=18,p(1) + f)$ & $3.94~(3.06,4.98)$ & $2.28~(1.76,2.82)$ & $9.24~(4.57,17.21)$\\

\rowcolor{green!18} 
POD-SINDy $(d=20,p(1) + f)$ & $3.59~(2.59,5.16)$ & $2.07~(1.51,2.96)$ & $7.89~(3.75,14.13)$\\


DMs-MVAR $(d=4, l=5)$ & $2.69~(1.40,4.10)$ & $1.56~(0.82,2.35)$ & $3.02~(0.92,5.70)$\\

\rowcolor{green!18} 
DMs-MVAR $(d=6, l=3)$ & $2.21~(1.24,3.22)$ & $1.28~(0.72,1.85)$ & $2.48~(0.97,4.35)$\\


\rowcolor{green!18}
DMs-SINDy $(d=4, p(3))$ & $2.43~(1.13,4.17)$ & $1.41~(0.66,2.40)$ & $2.05~(0.90,3.31)$\\

DMs-SINDy $(d=6, p(1) + f)$ & $2.69~(1.79,3.71)$ & $1.56~(1.03,2.13)$ & $4.72~(1.43,7.36)$\\

\bottomrule
\end{tabularx}
}
\caption{Fluid dynamics paradigm: summary of manifold-informed ROMs performance on the validation set \(X_{\text{val}}\). Columns report absolute $\ell_2$ error ($\varepsilon_2$), relative $\ell_2$ error ($\varepsilon_2^{\,r}$), and Wasserstein–1 distance ($W_1$) in terms of average over time and 10$^{\text{th}}$–90$^{\text{th}}$ percentiles. Rows shaded in green highlight the selected POD-informed and DMs-informed ROMs.}
\label{tab:rom_metrics_validation_new}
\end{table}

We detail here the procedure used to select the optimal reduced-order models (ROMs) for the
fluid dynamics paradigm. In particular, we investigate the dependence of reconstruction and
forecasting accuracy on the number of retained POD modes and DMs coordinates, and compare
the performance of the MVAR and SINDy learning algorithms applied to these latent
coordinates.

For the MVAR models, the lag order $l$ was selected using the Bayesian Information Criterion
(BIC). For the POD-informed models, the BIC selects $l=3$ for $d=16$ modes, while
$l=2$ for $d=18$ and $d=20$. For the DMs-informed models, the BIC selects $l=5$ for
the lower-dimensional embedding ($d=4$), and $l=3$ for $d=6$. Tables~\ref{tab:rom_metrics_train_new} and~\ref{tab:rom_metrics_validation_new} summarize ROM performance on the training and validation datasets, respectively, in
terms of absolute error ($\varepsilon_2$), relative error ($\varepsilon_2^r$), and
Wasserstein--1 distance ($W_1$), reported as time-averaged values with 10$^{\text{th}}$–90$^{\text{th}}$ percentiles. For the reconstruction of the DMs-informed ROMs predictions in the high-dimensional space, we have used the $k$-NN algorithm with $k=6$.

Within the DMs-informed ROMs, the MVAR reconstruction accuracy improves with the number
of retained coordinates, with the best validation performance achieved for $d=6$ and $l=3$
(Table~\ref{tab:rom_metrics_validation_new}). For DMs-informed SINDy, more parsimonious models yield better performance:
the cubic polynomial library $p(3)$ with $d=4$ coordinates attains the lowest $W_1$ on the
validation set, while increasing the latent dimension to $d=6$ (with a $p(1)+f$ library)
degrades accuracy.

For the POD-informed ROMs, MVARs exhibit increasing errors as $d$ grows from 16 to 20, indicating that
retaining additional POD modes does not translate into improved latent-space predictability
for the linear autoregressive model in this setting. Conversely, POD-informed
SINDy benefits from richer POD embeddings: increasing the number of modes reduces the
validation error, with the best performance obtained at $d=20$ using a first-order polynomial
library augmented with Fourier features ($p(1)+f$). Based primarily on the $W_1$ metric evaluated on both the training and validation datasets, we selected $d=6$ as the optimal DMs-informed MVAR ROM and $d=16$ as
the optimal POD-informed MVAR ROM, and $d=4$ as the optimal DMs-informed SINDy ROM and
$d=20$ as the optimal POD-informed SINDy ROM. Results obtained with these four selected
models have been presented and discussed in Section~\ref{subsec:results_fluid}.


\bibliography{mybibfile}
\bibliographystyle{ieeetr}

\end{document}